\documentclass[12pt]{article}
\usepackage{amsmath}
\usepackage[affil-it]{authblk}
\usepackage{amsfonts}
\usepackage{amsthm}
\usepackage{appendix}
\usepackage{algorithmic}
\usepackage{graphicx}
\usepackage{algorithm2e}
\usepackage{tabu}
\theoremstyle{definition}
\newtheorem{definition}{Definition}[section]
\usepackage{mathtools}

\theoremstyle{lemma}
\newtheorem{theorem}{Theorem}
\newtheorem{corollary}{Corollary}[theorem]
\newtheorem{lemma}[theorem]{Lemma}
\newtheorem{proposition}{Proposition}

\newcommand{\norm}{\text{Norm}}
\newcommand{\nrd}{\text{nrd}}
\newcommand{\trace}{\text{Trace}}

\usepackage[margin=0.5in]{geometry}
\begin{document}
\title{Reduction Theory of Algebraic Modules and their Successive Minima}
\author{Christian Porter\thanks{Department of Electrical and Electronic Engineering, Imperial College London, United Kingdom. Corresponding author, c.porter17@imperial.ac.uk}, Cong Ling\thanks{Department of Electrical and Electronic Engineering, Imperial College London, United Kingdom. c.ling@imperial.ac.uk}}
\maketitle
\begin{abstract}
    Lattices defined as modules over algebraic rings or orders have garnered interest recently, particularly in the fields of cryptography and coding theory. Whilst there exist many attempts to generalise the conditions for LLL reduction to such lattices, there do not seem to be any attempts so far to generalise stronger notions of reduction such as Minkowski, HKZ and BKZ reduction. Moreover, most lattice reduction methods for modules over algebraic rings involve applying traditional techniques to the embedding of the module into real space, which distorts the structure of the algebra.
    \\ In this paper, we generalise some classical notions of reduction theory to that of free modules defined over an order. Moreover, we extend the definitions of Minkowski, HKZ and BKZ reduction to that of such modules and show that bases reduced in this manner have vector lengths that can be bounded above by the successive minima of the lattice multiplied by a constant that depends on the algebra and the dimension of the module. In particular, we show that HKZ reduced bases are polynomially close to the successive minima of the lattice in terms of the module dimension. None of our definitions require the module to be embedded and thus preserve the structure of the module.
\end{abstract}
\section{Introduction}
Reduction theory of lattices is the study of representing the basis of a lattice in a manner such that the basis exhibits desirable properties, initially spurred by the study of the minima of positive definite quadratic forms and showing equivalence between two forms. There are a number of more precise definitions of what constitutes a reduced lattice basis, and perhaps the most widely known but distinct definitions of reduced bases are respectfully attributed to Lenstra-Lenstra-Lov\`asz, Korkin-Zolotarev and Minkowski (see \cite{LLL}, \cite{KZ}, \cite{quadratische}). The study of lattice reduction has gained traction in recent years due to the rise of cryptosystems based on lattice problems (see e.g. \cite{knapsack}, \cite{knapsack2}), and also in coding theory (see e.g. \cite{spheredecode}). Lately, generalisations of real lattices to those spanned over algebraic fields have emerged as a contender for classical lattices, both in cryptography for their relative compactness in terms of key size required to define them \cite{keysize} and in coding theory for the fine structure of lattices defined over such algebras \cite{rayleigh}. Whilst ``weak'' definitions of reduced lattices, such as LLL reduced lattices, have been extended to that of their algebraic counterpart (for just a few examples in literature, see \cite{euclidean}, \cite{latticenumberfield}, \cite{shortnumberfield}, \cite{napias}), similar research into ``strong'' definitions of reduced lattice bases is somewhat limited. In this work, we establish the algebraic counterpart of some classical notions of lattice reduction, namely Minkowski, HKZ and BKZ reduced bases, and taking inspiration from \cite{quadratische2}, \cite{HKZ} and \cite{schnorr} we prove that bases reduced in this manner exhibit properties that are deemed desirable in reduction theory.
\section{Preliminaries}
We begin by defining some familiar concepts in algebraic number theory and lattice theory. For any concepts concerning central division algebras that we have left unexplained, we refer the reader to \cite{maxorders}. Moreover, throughout the paper we assume all division algebras in question are central division algebras. Denote by $K$ some division algebra, and let $\mathcal{O}$ be some order of $K$. We say that $\mathcal{O}$ is \emph{left-Euclidean} (respectively \emph{right-Euclidean}) if there exists a function $\phi: \mathcal{O} \to \mathbb{Z}^+$ such that, for all $a,b \in \mathcal{O}$, there exist $q,r \in \mathcal{O}$ such that $a=qb+r$ for some $q,r \in \mathcal{O}$, $r=0$ or $0<\phi(r)<\phi(b)$ (respectively, we take $bq$ instead of $qb$ is the order is right-Euclidean). In this paper, we will only cover associative division algebras, and so multiplicative operations throughout the paper will be assumed to be associative. The following definition of a lattice will be used throughout the rest of this piece of work. Here, by $x \mathbf{v}$ where $x \in K$, $\mathbf{v}=(v_1,\dots,v_D) \in K^D$, multiplication by a vector is defined componentwise, where the direction of multiplication is defined by the position of $x$, i.e. $x\mathbf{v}=(xv_1,xv_2,\dots,xv_D)$, $\mathbf{v}x=(v_1 x, v_2 x, \dots, v_D x)$.
\begin{definition}
Let $K$ be a division algebra, and $\mathcal{O}$ some order of $K$. Suppose that $\Lambda$ is a left- (or right-) module over $\mathcal{O}$. We say that $\Lambda$ is a \emph{left-sided lattice} (or respectively \emph{right-sided lattice}) of dimension $d$ if $\Lambda$ has the representation
\begin{align*}
    \Lambda= \bigoplus_{i=1}^d\mathcal{O} \mathbf{b}_i,
\end{align*}
(respectively $\mathbf{b}_i\mathcal{O}$ if $\Lambda$ is a right-sided lattice), where $\mathbf{b}_i \in K^D$ for some integers $1 \leq d \leq D$, and each $\mathbf{b}_i$ is linearly independent over $K$. The set $B=\{\mathbf{b}_1,\dots,\mathbf{b}_d\}$ is said to be the basis of $\Lambda$.
\end{definition}
From now on, we will only refer to left-sided lattices, and we will refer to them simply as ``lattices'' unless we need to specify otherwise.
\begin{definition}
Let $K$ be a division algebra and denote by $\Lambda, \Lambda^\prime$ lattices spanned over an order $\mathcal{O}$ of $K$ with bases $B=\{\mathbf{b}_1,\dots,\mathbf{b}_d\},B^\prime=\{\mathbf{b}_1^\prime, \dots, \mathbf{b}_d^\prime\}$, respectively. We say that $\Lambda,\Lambda^\prime$ are \emph{equivalent} if the two modules are isomorphic, that is, for every element $\mathbf{v} \in \Lambda$, $\mathbf{v}$ is also contained in $\Lambda^\prime$. We say that a set $S=\{\mathbf{s}_1,\dots,\mathbf{s}_k\}, \mathbf{s}_i \in K^D$ for some $1 \leq k \leq d \leq D$ and $S$ is linearly independent over $K$, is \emph{extendable to a basis} for $\Lambda$ if there exists an equivalent lattice $\Lambda^\prime $ with basis $B^\prime $ where $S \subseteq B^\prime $.
\end{definition}
\begin{proposition}\label{equivalent}
Let $\Lambda, \Lambda ^\prime$ be equivalent lattices, and assume $B^\prime$ is a basis for $\Lambda^\prime$. Then $B ^\prime$ is also a basis for $\Lambda$.
\end{proposition}
\begin{proof}
This follows from the definition of equivalent lattices, since any $\mathbf{v} \in \Lambda ^\prime$ is also in $\Lambda$, and so the basis describing $\Lambda ^\prime$ also describes $\Lambda$.
\end{proof}
\begin{proposition}
Let $\Lambda$ be a $d$-dimensional lattice spanned over an order $\mathcal{O}$ of a division algebra $K$ with unit group $\mathcal{O}^*$, and let $B=\{\mathbf{b}_1,\dots,\mathbf{b}_d\}$ be a basis for $\Lambda$. Then, for any $1 \leq k \leq d$, $x_i \in \mathcal{O}$ and $u \in \mathcal{O}^*$, $\left\{\mathbf{b}_1,\dots, \mathbf{b}_{k-1},u\mathbf{b}_k+\sum_{i=1: i \neq k}^dx_i\mathbf{b}_i,\mathbf{b}_{k+1},\dots,\mathbf{b}_d\right\}$ is extendable to a basis for $\Lambda$.
\end{proposition}
\begin{proof}
Let $\Lambda^\prime$ be the lattice with the basis $\left\{\mathbf{b}_1,\dots, \mathbf{b}_{k-1},u\mathbf{b}_k+\sum_{i=1: i \neq k}^dx_i\mathbf{b}_i,\mathbf{b}_{k+1},\dots,\mathbf{b}_d\right\}$. Then, for some $y_1,\dots,y_d \in \mathcal{O}$, we have
\begin{align*}
    \sum_{i=1: i \neq k}^dy_i\mathbf{b}_i+y_k\left(u\mathbf{b}_k+\sum_{i=1: i \neq k}^dx_i\mathbf{b}_i\right)=\sum_{i=1: i \neq k}^d (y_i+y_kx_i)\mathbf{b}_i+(y_ku)\mathbf{b}_k \in \Lambda,
\end{align*}
and for any $z_1,\dots,z_d \in \mathcal{O}$, we have
\begin{align*}
    \sum_{i=1}^dz_i\mathbf{b}_i=\sum_{i=1: i\neq k}^dz_i\mathbf{b}_i+z_k\mathbf{b}_k=\sum_{i=1: i \neq k}^d\left(z_i-z_ku^{-1}x_i\right)\mathbf{b}_i+(z_ku^{-1})\left(u\mathbf{b}_k+\sum_{i=1: i \neq k}^dx_i\mathbf{b}_i\right) \in \Lambda^\prime,
\end{align*}
and so we have shown the lattices are equivalent, as required.
\end{proof}
For a division algebra $K$, and a set $S=\{\mathbf{s}_1,\dots,\mathbf{s}_k\}, \mathbf{s}_i \in K^D$ for some integers $D,k$, for all $1 \leq i \leq D$, define by $\text{Span}_K(S)$ the space of all linear combinations of $S$ over $K$, where multiplication by scalars is performed on the lefthand side of vectors.
\begin{definition}
Let $\Lambda$ be a lattice of dimension $d$ over $\mathcal{O}$, some order of a division algebra $K$, and suppose that $S=\{\mathbf{s}_1,\dots,\mathbf{s}_k\}$ is a set of elements of $\Lambda$ that are linearly independent over $K$, $k \leq d$. We say that $S$ is a \emph{primitive system} of $\Lambda$ if, for all $\mathbf{y} \in \Lambda$, then $\mathbf{y} \in \text{Span}_K(S)$ if and only if $\mathbf{y}=\sum_{i=1}^k x_i \mathbf{s}_i$, where $x_i \in \mathcal{O}$.
\end{definition}
\begin{proposition}\label{primitive}
Let $K$ be a division algebra, and let $\mathcal{O}$ be a right-Euclidean order of $K$, and say that $\Lambda$ is a lattice of dimension $d$ defined over $\mathcal{O}$ with basis $B=\{\mathbf{b}_1,\dots,\mathbf{b}_d\}$. Then for any set of $k$ linearly independent vectors $S=\{\mathbf{s}_1,\dots,\mathbf{s}_k\}$ in $K^D$, $1 \leq k \leq d \leq D$, $S$ is a primitive system of $\Lambda$ if and only if $S$ is extendable to a basis of $\Lambda$.
\end{proposition}
\begin{proof}
The if statement is trivial: if $S$ is a subset of a basis of an equivalent lattice $\Lambda^\prime$, then by the definition of a lattice over $\mathcal{O}$, the linear span of $S$ over $K$ is an element of the lattice only if the elements taken from $K$ are in $\mathcal{O}$. Now assume that $S$ is a primitive system of $\Lambda$. We prove by induction, and so begin by taking $k=1$. Let $S=\{\mathbf{s}\}$. Since $\mathbf{s}$ is an element of $\Lambda$, we may use the representation
\begin{align*}
    \mathbf{s}=\sum_{i=1}^dx_i\mathbf{b}_i,
\end{align*}
where $x_i \in \mathcal{O}$. Assume that $x_1$ is nonzero, and the smallest nonzero element when ordering the $x_i$ in terms of the Euclidean function $\phi$. If $x_2=x_3=\dots=x_d=0$, then we have
\begin{align*}
    \mathbf{s}=x_1\mathbf{b}_1.
\end{align*}
However, by the definition of a primitive system, this can only be true if we have $x_1 \in \mathcal{O}^*$, the unit group of $\mathcal{O}$. Now assume there is at least one such $x_j$ where $x_j \neq 0, j \neq 1$. Assume each $x_j, j \geq 2$ is such that $\phi(x_1)\leq \phi(x_j)$, except for those of the $x_j=0$. Then, by the definition of a Euclidean ring, we may choose a $q_j$ for each $x_j$ such that either $x_j-x_1q_j=0$ or $\phi(x_j-x_1q_j)<\phi(x_1)$, and so we have
\begin{align*}
    \mathbf{s}= \sum_{i=1}^dx_i\mathbf{b}_i=x_1\left(\mathbf{b}_1+\sum_{i=2}^dq_i\mathbf{b}_i\right)+\sum_{i=2}^d(x_i-x_1q_i)\mathbf{b}_i=x_1\mathbf{b}_1^*+ \sum_{i=2}^dx_i^* \mathbf{b}_i,
\end{align*}
where $\mathbf{b}_1^*=\left(\mathbf{b}_1+\sum_{i=2}^dq_i\mathbf{b}_i\right)$ and $x_i^*=(x_i-x_1q_i)$. Since every $x_i^*$ is such that $\phi(x_i^*)<\phi(x_1)$, and $x_1$ is the smallest nonzero element in terms of the Euclidean function $\phi$, iterating this procedure a finite number of times yields
\begin{align*}
    \mathbf{s}=m\mathbf{b}^\prime,
\end{align*}
where $\mathbf{b}^\prime$ is some lattice vector achieved by invertible operations. By definition of a primitive system, we must have $m \in \mathcal{O}^*$, and as such $S$ is extendable to a basis of $\Lambda$, as we have come by $\mathbf{b}^\prime$ using invertible operations. Now, assume that $S_{k+1}=\{\mathbf{s}_1,\dots,\mathbf{s}_k,\mathbf{s}_{k+1}\}$ is a primitive system, and $S_k=\{\mathbf{s}_1,\dots,\mathbf{s}_k\}$ is extendable to a basis of $\Lambda$. Let $B=\{\mathbf{b}_1,\dots,\mathbf{b}_d\}$ be a basis for $\Lambda$. By proposition \ref{equivalent} and the assumption that $S_k$ is extendable to a basis for $\Lambda$, we may set $\mathbf{b}_1=\mathbf{s}_1,\dots,\mathbf{b}_k=\mathbf{s}_k$. We may use the representation
\begin{align*}
    \mathbf{s}_{k+1}=\sum_{i=1}^kx_i\mathbf{s}_i+\sum_{i=k+1}^dx_i\mathbf{b}_i,
\end{align*}
for some $x_i \in \mathcal{O}$, where at least one of $x_{k+1},\dots,x_d$ is nonzero as the set $S_{k+1}$ is linearly independent over $K$. Using Euclidean division with the coefficients $x_{k+1},\dots,x_d$ as before, we get
\begin{align*}
    \mathbf{s}_{k+1}=\sum_{i=1}^kx_i\mathbf{s}_i+m\mathbf{b}_{k+1}^*.
\end{align*}
We must have $m \in \mathcal{O}^*$, as $\mathbf{b}_{k+1}^*$ is a lattice vector, and if $m \notin \mathcal{O}^*$, then
\begin{align*}
    \mathbf{b}_{k+1}^*=m^{-1}\left(\mathbf{s}_{k+1}-\sum_{i=1}^kx_i\mathbf{s}_i\right) \notin \Lambda,
\end{align*}
which is a contradiction. Therefore, since we have come by the vector $\mathbf{b}_{k+1}^*$ using invertible operations, we have shown the set $S_{k+1}$ is extendable to a basis for $\Lambda$.
\end{proof}
The following corollaries are an immediate consequence of this proposition.
\begin{corollary}\label{extendable}
Let $\Lambda$ be a lattice of dimension $d$ spanned over a right-Euclidean order $\mathcal{O}$ of some division algebra $K$ with basis $B=\{\mathbf{b}_1,\dots,\mathbf{b}_d\}$. Denote by $\mathbf{v}=\sum_{i=1}^dx_i\mathbf{b}_i$ an arbitrary lattice vector such that $x_i \in \mathcal{O}$, $\gcd(x_1,\dots,x_d)=1$. Then there exists a set containing $\mathbf{v}$ that is extendable to a basis for $\Lambda$.
\end{corollary}
\begin{corollary}\label{GCD}
Let $\Lambda$ be a lattice of dimension $d$ spanned over a right-Euclidean order $\mathcal{O}$ of some division algebra $K$, with basis $B=\{\mathbf{b}_1,\dots,\mathbf{b}_d\}$. Denote by
\begin{align*}
    \mathbf{s}_{k}=\sum_{i=1}^dx_i\mathbf{b}_i
\end{align*}
an element of $\Lambda$, so $x_1,\dots, x_d \in \mathcal{O}$. Then, for all $1\leq k \leq d$, the set $S_k=\{\mathbf{b}_1,\dots,\mathbf{b}_{k-1},\mathbf{s}_k\}$ forms a primitive system if and only if at least one of $x_k,\dots,x_d$ is nonzero, and $\gcd(x_k,\dots,x_d)=1$.
\end{corollary}
\section{Minkowski Reduction of Algebraic Lattices}
In order to ascertain more important properties about algebraic lattices we need to define how we measure the lengths of lattice vectors, and what it means for a lattice basis to be reduced with respect to the norm function.
\begin{definition}
Let $\Lambda$ be a lattice over an order $\mathcal{O}$ of some division algebra $K$. A function $N: \Lambda \to \mathbb{R}^+ \cup \{0\}$, $N: \mathcal{O} \to \mathbb{R} \cup \{0\}$ is a \emph{norm} on $\Lambda$ if it satisfies the following properties:
\begin{itemize}
    \item $N(\mathbf{v}+\mathbf{w}) \leq N(\mathbf{v})+ N(\mathbf{w})$, for all $\mathbf{v},\mathbf{w} \in \Lambda$,
    \item $N(a\mathbf{v}) \leq N(a)N(\mathbf{v})$, for all $\mathbf{v} \in \Lambda, a \in \mathcal{O}$,
    \item $N(\mathbf{v})=0 \iff \mathbf{v}$ is the zero vector.
\end{itemize}
\end{definition}
\begin{definition}
Let $\Lambda$ be a $d$-dimensional lattice with basis $B=\{\mathbf{b}_1,\dots,\mathbf{b}_d\}$ over an order $\mathcal{O}$ of some division algebra $K$, and let $N$ be a norm on $\Lambda$. We say that $\Lambda$ is \emph{Minkowski reduced} if $B$ satisfies the following properties:
\begin{itemize}
    \item $\mathbf{b}_1$ is the smallest, nonzero vector with respect to $N$ such that $\{\mathbf{b}_1\}$ is extendable to a basis for $\Lambda$,
    \item For all $1 \leq k \leq d$, $\mathbf{b}_k$ is the shortest nonzero vector with respect to $N$ such that $\{\mathbf{b}_1,\dots,\mathbf{b}_k\}$ is extendable to a basis for $\Lambda$.
\end{itemize}
\end{definition}
\begin{definition}
Let $\Lambda$ be a $d$-dimensional lattice with basis $B=\{\mathbf{b}_1,\dots,\mathbf{b}_d\}$ over an order $\mathcal{O}$ of some division algebra $K$, and let $N$ be a norm on $\Lambda$. Denote by $V=\{\mathbf{v}_1,\dots,\mathbf{v}_d\}$ the elements of $\Lambda$ such that \begin{itemize}
    \item $\text{Span}_K(B)=\text{Span}_K(V)$,
    \item $N(\mathbf{v}_1) \leq N(\mathbf{v}_2) \leq \dots \leq N(\mathbf{v}_d)$,
    \item For every linearly independent set $W=\{\mathbf{w}_1,\dots,\mathbf{w}_d\}$ over $K$, $\mathbf{w}_i \in \Lambda$, we have $N(\mathbf{v}_i) \leq N(\mathbf{w}_i)$, for all $1 \leq i \leq d$.
\end{itemize}
We label $N(\mathbf{v}_1)=\lambda_1,\dots, N(\mathbf{v}_d)=\lambda_d$. Then $\lambda_i$ are referred to as the \emph{$i$th successive minima of $\Lambda$}, for all $1 \leq i \leq d$.
\end{definition}
\begin{theorem}\label{minkowskigcd}
Let $\Lambda$ be a $d$-dimensional lattice with basis $B=\{\mathbf{b}_1,\dots,\mathbf{b}_d\}$ over a right-Euclidean order $\mathcal{O}$ of some division algebra $K$, and let $N$ be a norm on $\Lambda$. Then $\Lambda$ is Minkowski reduced if and only if, for all $1 \leq k \leq d$, $x_1,\dots,x_d \in \mathcal{O}$, the following implications hold:
\begin{align*}
    N\left(\sum_{i=1}^dx_i\mathbf{b}_i\right) \geq N(\mathbf{b}_k) \iff \gcd(x_k,\dots,x_d)=1.
\end{align*}
\end{theorem}
\begin{proof}
This follows from Proposition \ref{primitive} and Corollary \ref{GCD}.
\end{proof}
\subsection{The quadratic norm}
Let $K$ be a division algebra of degree $m$ over some base field $F$, and suppose that $F$ is of degree $n=r_1+2r_2$ over $\mathbb{Q}$, where $r_1$ is the number of real places and $r_2$ is the number of pairs of complex places. Denote by $\mathbb{H}$ the Hamilton quaternion field over $\mathbb{R}$, and denote $K_{\mathbb{R}}=K \otimes_{\mathbb{Q}} \mathbb{R}$. Denote by $\varphi$ the homomorphism $\varphi: K \to K_\mathbb{R}$. We have
\begin{align*}
    K_\mathbb{R} \cong \mathbb{M}_{\frac{m}{2}}(\mathbb{H}) \times \mathbb{M}_m(\mathbb{R})^{r_1-w} \times \mathbb{M}_m(\mathbb{C})^{r_2},
\end{align*}
where $w$ is the number of real places at which $K$ ramifies. We define by $*$ the canonical involution of $K_{\mathbb{R}}$, which is induced by the canonical involution of the quaternion field on the first factor, the identity map on the second and complex conjugation on the third. It follows that, for any $x,y \in K_\mathbb{R}$, we have $(xy)^*=y^*x^*$. Associate to $K$ the reduced trace function $\text{tr}: K_{\mathbb{R}} \to \mathbb{R}$. Then, we define the following bilinear form, for all $\mathbf{x}=(x_1,\dots,x_D),\mathbf{y}=(y_1,\dots,y_D) \in K^m$:
\begin{align*}
    q_\alpha(\varphi(\mathbf{x}),\varphi(\mathbf{y}))=\text{tr}_{K/\mathbb{Q}}\left(\sum_{i=1}^D\varphi(x_i)\varphi(\alpha) \varphi(y_i)^*\right),
\end{align*}
where $\text{tr}$ denotes the reduced trace induced by the division algebra, and $\alpha \in K$ such that $\text{tr}_{K/\mathbb{Q}}(\varphi(x)\varphi(\alpha)\varphi(x)^*)$ induces a positive-definite quadratic form for all $x \in K$, and $\varphi(\mathbf{x})=(\varphi(x_1),\varphi(x_2),\dots, \varphi(x_D))$. We remark that $q_\alpha$ is positive-definite if and only if $\varphi(\alpha)=\varphi(\alpha)^*$ and $\alpha \neq 0$, and we say that $\alpha$ is \emph{totally positive} if it satisfies this property. We denote by $\mathcal{P}$ the subset of totally positive elements of $K$.  For any $a,b \in K$, by abuse of notation we also let $q_\alpha(a,b)=\text{tr}_{K/\mathbb{Q}}(\varphi(a) \varphi(\alpha) \varphi(b)^*)$. We note that the square root of $q_\alpha$ acts as a norm for any lattice $\Lambda$ of dimension $d \leq D$  over an order $\mathcal{O}$ of $K$ $\alpha \in \mathcal{P}$. For convenience of notation, we will write $q_\alpha(\varphi(\mathbf{x})):=q_\alpha(\varphi(\mathbf{x}),\varphi(\mathbf{x}))$ for any $\mathbf{x} \in K^D$. Moreover, for any $\mathbf{x}=(x_1,\dots,x_D),\mathbf{y}=(y_1,\dots,y_D) \in K^D$ let us denote by $\langle \cdot, \cdot \rangle_\alpha: K_\mathbb{R}^D \times K_{\mathbb{R}}^D \to K_\mathbb{R}$, $\langle \varphi(\mathbf{x}), \varphi(\mathbf{y}) \rangle_\alpha:=\sum_{i=1}^D \varphi(x_i)\varphi(\alpha)\varphi(y_i)^* \in K_\mathbb{R}$. Then for some arbitrary lattice vector $\mathbf{v}=\sum_{i=1}^dx_i\mathbf{b}_i$ for a lattice $\Lambda \subset K^D$, where $\mathbf{b}_i=(b_{1,i},b_{2,i},\dots,b_{D,i})$, $\mathbf{b}_i \in K^D$, we have
\begin{equation}\label{orthogonal}
    q_\alpha(\varphi(\mathbf{v}))=q_\alpha\left(\sum_{i=1}^d\varphi(x_i)\varphi(\mathbf{b}_i)\right)=q_\alpha\left(\sum_{i=1}^d\varphi(x_i)\left(\mathbf{b}_i(i)+\sum_{j=1}^{i-1}\mu_{i,j}\mathbf{b}_j(j)\right)\right),
\end{equation}
where $\mathbf{b}_i(i)=\varphi(\mathbf{b}_i)-\sum_{j=1}^{i-1}\mu_{i,j}\mathbf{b}_j(j)$, $\mathbf{b}_1(1)=\varphi(\mathbf{b}_1)$ and $\mu_{i,j}=(\langle\varphi(\mathbf{b}_i),\mathbf{b}_j(j)\rangle_\alpha)(\langle \mathbf{b}_j(j),\mathbf{b}_j(j) \rangle_\alpha )^{-1} \in K_\mathbb{R}$, for all $1 \leq j <i \leq d$.
\begin{lemma}
For all $1 \leq j < i \leq d$, we have
\begin{align*}
    q_\alpha(\varphi(x_i)\mathbf{b}_i(i),\varphi(x_j)\mathbf{b}_j(j))=0,
\end{align*}
where $x_i,x_j \in K$.
\end{lemma}
\begin{proof}
The proof follows closely to that of classical Gram-Schmidt orthogonalisation. First, let's show the claim for $j=1,i=2$. Letting $\mathbf{b}_1=(b_{1,1},b_{2,1},\dots,b_{D,1}), \mathbf{b}_2=(b_{1,2},b_{2,2},\dots, b_{D,2})$:
\begin{align*}
    &q_\alpha(\varphi(x_2)\mathbf{b}_2(2),\varphi(x_1)\mathbf{b}_1(1))=\text{tr}_{K/\mathbb{Q}}\left(\langle \varphi(x_2)\varphi(\mathbf{b}_2),\varphi(x_1)\mathbf{b}_1\rangle_\alpha -\langle \varphi(x_2)\mu_{21}\mathbf{b}_1(1),\varphi(x_1)\mathbf{b}_1(1)\rangle_\alpha\right)
    \\&=\text{tr}_{K/\mathbb{Q}}\left(\sum_{i=1}^D\left(\varphi(x_2)\varphi(b_{i,2})\varphi(\alpha)(\varphi(x_1)\varphi(b_{i,1}))^*-\varphi(x_2)\mu_{21}\varphi(b_{i,1})\varphi(\alpha)(\varphi(x_1)\varphi(b_{i,1}))^*\right)\right)
    \\&=\text{tr}_{K/\mathbb{Q}}\left(\sum_{i=1}^D\varphi(x_2)\varphi(b_{i,2})\varphi(\alpha)\varphi(b_{i,1})^*\varphi(x_1)^*-\varphi(x_2)\mu_{21}\sum_{i=1}^D(\varphi(b_{i,1})\varphi(\alpha)\varphi(b_{i,1})^*)\varphi(x_1)^*\right),
\end{align*}
which is zero, by the definition of $\mu_{21}$. A similar proof follows if we replace $2$ with an arbitrary $i$. Now suppose that $q_\alpha(\varphi(x_i)\mathbf{b}_i(i),\varphi(x_j)\mathbf{b}_j(j))=0$ for all $j  <  k<i-1$. Then for some $x_{k+1} \in K$:
\begin{align*}
    &q_\alpha(\varphi(x_i)\mathbf{b}_i(i),\varphi(x_{k+1})\mathbf{b}_{k+1}(k+1))\\&=q_\alpha(\varphi(x_i)\mathbf{b}_i(i),\varphi(x_{k+1})\varphi(\mathbf{b}_{k+1}))-\sum_{l=1}^{k}q_\alpha(\varphi(x_i)\mathbf{b}_i(i),\varphi(x_{k+1})\mu_{i,l}\mathbf{b}_l(l))\\&=q_\alpha(\varphi(x_i)\mathbf{b}_i(i),\varphi(x_{k+1})\varphi(\mathbf{b}_{k+1})),
\end{align*}
and so the above $=0$ by an identical argument for the case $i=2,j=1$.
\end{proof}
\begin{definition}
Associate to $K$ the reduced norm function $\nrd_{K/\mathbb{Q}}:K_\mathbb{R} \to \mathbb{R}$. Let $\Phi$ be the homomorphism that takes $K \to M_{m \times m}(F)$. Let $\Lambda$ be a lattice spanned over an order $\mathcal{O}$ of $K$ in the space $K^D$ and basis $B=\{\mathbf{b}_1,\dots,\mathbf{b}_d\}$. We define the determinant of $\Lambda$ by
\begin{align*}
    \det_\alpha(\Lambda):=\prod_{i=1}^r \det(\mathcal{F}_i)\prod_{i=r}^{r+s}\det(\mathcal{F}_i)^2,
\end{align*}
where, if $f_i$ denotes the quadratic form generated by $f_i(x_1,\dots,x_d)=\sum_{k=1}^D\varphi(\sigma_i(v_k))\varphi(\sigma_i(\alpha))\varphi(\sigma_i(v_k))^*=\sum_{j,k=1}^dx_jf_i^{j,k}x_k^*$, $\mathbf{v}=\sum_{i=1}^dx_i\mathbf{b}_i=(v_1,\dots,v_D)$, $f_{j,k}=f_{k,j}^*$, where we associate $r$ real automorphisms and $s$ pairs of complex automorphisms $\sigma_i$ to $F$, then $\mathcal{F}_i$ denotes the matrix made up of the submatrices $\Phi(f_i^{j,k})$. Then we define the \emph{additive $(\alpha,\mathcal{O},d)$--Hermite invariant} of an algebraic lattice $\Lambda$ of dimension $d$ over an order $\mathcal{O}$ by
\begin{align*}
    \gamma_{\alpha,\mathcal{O},d}(\Lambda)=\frac{\min_{\mathbf{v} \in \Lambda \setminus \{\mathbf{0}\}}q_\alpha(\varphi(\mathbf{v}))}{\det_\alpha(\Lambda)^{\frac{1}{dnm^2}}}.
\end{align*}
\end{definition}
Since $q_\alpha$ induces a positive-definite quadratic form, the value of the additive $(\alpha,\mathcal{O},d)$--Hermite invariant is bounded for every lattice $\Lambda$. We call $\gamma_{\alpha,\mathcal{O},d}=\max_{\Lambda} \gamma_{\alpha,\mathcal{O},d}(\Lambda)$ the \emph{additive $(\alpha,\mathcal{O},d)$--Hermite constant}. The following theorem can be proven identically to the case in \cite{leibak}.
\begin{theorem}
Denote by $\gamma_k$ the real Hermite constant in dimension $k$. Then for all positive $\alpha \in \mathcal{P}$,
\begin{align*}
    \gamma_{\alpha,\mathcal{O},d} \leq \gamma_{dnm^2}\text{disc}(\mathcal{O}/\mathbb{Z})^{\frac{1}{nm^2}},
\end{align*}
where $\text{disc}(\mathcal{O}/\mathbb{Z})$ denotes the discriminant of $\mathcal{O}$ over $\mathbb{Z}$.
\end{theorem}
\begin{theorem}
Let $\Lambda \subset K^D$ be a lattice of dimension $d \leq D$ spanned over the order $\mathcal{O}$, and let $q_\alpha$ be a quadratic norm defined by some $\alpha \in \mathcal{P}$. Denote by $\lambda_i$ the $i$th successive minima, $1 \leq i \leq d$, with respect to the norm $\sqrt{q_\alpha(\varphi(\mathbf{v}))}$, for all $\mathbf{v} \in K^D$. Then
\begin{align*}
    \prod_{i=1}^d\lambda_i^{2} \leq \gamma_{\alpha,\mathcal{O},d}^d\det_\alpha(\Lambda)^{\frac{1}{nm^2}}.
\end{align*}
\end{theorem}
\begin{proof}
Let $\mathbf{s}_1,\dots,\mathbf{s}_d$, $\mathbf{s}_i \in K^D$ be the vectors such that $q_{\alpha}(\varphi(\mathbf{s}_i))=\lambda_i^2$ for all $1 \leq i \leq d$. By the definition of the successive minima of the lattice, $\mathbf{s}_1,\dots,\mathbf{s}_d$ are linearly independent over $K$, and so every lattice point can be represented by
\begin{align*}
    \mathbf{v}=\sum_{i=1}^d x_i\mathbf{s}_i,
\end{align*}
for some $x_i \in K$. Using the method described in (\ref{orthogonal}), we may decompose $q_\alpha(\varphi(\mathbf{v}))$ into the sum of squares:
\begin{align*}
    q_\alpha(\varphi(\mathbf{v}))=f_1(x_1,x_2,\dots,x_d)^2+f_2(x_2,x_3,\dots,x_d)^2+\dots +f(x_d)^2,
\end{align*}
where $f_i: K \to \mathbb{R}$. Now, let us consider an alternative lattice $\Lambda^\prime$, whose vector lengths generate the quadratic form
\begin{align*}
    q_\alpha(\mathbf{v}^\prime)=\lambda_1^{-2}f_1(x_1,x_2,\dots,x_d)^2+\lambda_2^{-2}f_2(x_2,x_3,\dots,x_d)^2+\dots +\lambda_d^{-2}f(x_d)^2,
\end{align*}
for some $\mathbf{v}^\prime \in \Lambda^\prime$. We claim that every nonzero element of $\Lambda^\prime$ has norm $ \geq 1$. Suppose that $f_k(x_k,\dots,x_d)$ is the first nonzero value in $f_1,f_2,\dots,f_d$, counting backwards from $d$. We must have that $\mathbf{v}$ is linearly independent of the vectors $\mathbf{s}_1,\dots,\mathbf{s}_{k-1}$, as otherwise, by definition we must have that $x_k=x_{k+1}=\dots=x_d=0$, and hence $f_k(x_k,\dots,x_d)=f_{k+1}(x_{k+1},\dots,x_d)=\dots=f_d(x_d)=0$, which is a contradiction. It therefore holds that $\sum_{i=1}^k f(x_i,\dots,x_d) \geq \lambda_k^2$, and so
\begin{align*}
    \sum_{i=1}^k\lambda_i^{-2}f(x_i,\dots,x_d) \geq \lambda_k^{-2}\sum_{i=1}^kf(x_i,\dots,x_d) \geq 1.
\end{align*}
Since $\det_{\alpha}(\Lambda^\prime)=\det_\alpha(\Lambda)\prod_{i=1}^d\lambda_i^{-2nm^2}$, using the fact that the shortest nonzero vector in $\Lambda^\prime$ has norm $1$, we get
\begin{align*}
    \gamma_{\alpha,\mathcal{O},d}^d\det_\alpha(\Lambda^\prime)^{\frac{1}{nm^2}} \geq 1 \iff \gamma_{\alpha,\mathcal{O},d}^d\det_\alpha(\Lambda)^{\frac{1}{nm^2}} \geq \prod_{i=1}^d\lambda_i^{2},
\end{align*}
which proves the result.
\end{proof}
\begin{lemma}[\cite{cdaminimum}, Corollary 3.9]
Let $K$ be a division algebra of degree $m$ over some base number field $F$, and let $[F:\mathbb{Q}]=n$. Denote by $\mathcal{O}$ some order of $F$ with discriminant $\text{disc}(\mathcal{O})$ over $\mathbb{Z}$. Then for any $x \in K_\mathbb{R}$, $\alpha \in \mathcal{P}$, there exists a $y \in \mathcal{O}$ such that
\begin{align*}
    \max_{x \in K}\min_{y \in \mathcal{O}}q_\alpha(x-\varphi(y)) \leq \frac{nm^3}{4}\text{disc}(\mathcal{O}/\mathbb{Z})^{\frac{2}{nm^2}}\text{nrd}_{K/\mathbb{Q}}(\alpha)^{\frac{1}{nm}}.
\end{align*}
\end{lemma}
From now on, in order to keep our notation concise, we will use the symbol $\rho_{\alpha,\mathcal{O}}$ to denote the quantity $\frac{nm^3}{4}\text{disc}(\mathcal{O}/\mathbb{Z})^{\frac{2}{nm^2}}\text{nrd}_{K/\mathbb{Q}}(\alpha)^{\frac{1}{nm}}$. Finally, we give the following definition:
\begin{definition}
Let $\mathcal{O}^\times$ denote the unit group of $\mathcal{O}$. We say that the space $K^D$ is \emph{left unit reducible} (respectively \emph{right unit reducible} with respect to a norm $N$ if, for any $\mathbf{v} \in K^D$, the following implications hold:
\begin{align*}
    N(\mathbf{v}) \leq N(u\mathbf{v}), \hspace{1mm} \forall u \in \mathcal{O}^\times  \iff N(\mathbf{v}) \leq N(q\mathbf{v}), \hspace{1mm} \forall q \in \mathcal{O} \setminus \{0\},
\end{align*}
(respectively right-hand multiplication for right unit reducible spaces). We say the space is unit reducible if the space is both left- and right- unit reducible.
\end{definition}
It is not currently clear which fields admit the unit reducible property and which do not, although we suspect that the value of the regulator of the number field would give a good indication of which fields are unit reducible or not. The property certainly holds for certain cases, and one can find counterexamples for certain fields. Clearly, the rational numbers $\mathbb{Q}$ and any imaginary quadratic or rational quaternion field admits a unit reducible space for any $D \geq 1$. We give a few examples of unit reducible spaces, and one counterexample.
\begin{proposition}\label{unitreduciblefields}
Let $N$ denote the quadratic norm $q_\alpha$ for any totally positive $\alpha \in K$. For any integer $D \geq 1$, the space $K^D$ is unit reducible if 
\begin{align*}
    K=\mathbb{Q}(\sqrt{2}),\mathbb{Q}(\sqrt{3}),\mathbb{Q}(\sqrt{5}),\mathbb{Q}(\zeta_8),\mathbb{Q}(\zeta_{12}).
\end{align*}
The space $K^D$ is not unit reducible if $K=\mathbb{Q}(\sqrt{6})$.
\end{proposition}
\begin{proof}
See appendices.
\end{proof}
We are now equipped to prove some useful properties about Minkowski reduced bases.
\begin{theorem}\label{minkowskibounds1}
Let $\Lambda$ be a lattice of dimension $d \leq D$ spanned over a right-Euclidean order $\mathcal{O}$, and assume that $K^D$ is left unit reducible. Denote by $B=\{\mathbf{b}_1,\dots,\mathbf{b}_d\}$, $\mathbf{b}_i \in K^D$ the basis vectors for $\Lambda$. Assume that $\Lambda$ is Minkowski reduced with respect to the norm induced by $q_{\alpha}$, for some $\alpha \in \mathcal{P}$. Denote by $\lambda_1,\dots,\lambda_d$ the successive minima of the lattice with respect to this norm. Then, for all $1 \leq k \leq d$, we have
\begin{align*}
    q_{\alpha}(\varphi(\mathbf{b}_k)) \leq \delta_k^2\lambda_k^2,
\end{align*}
where
\begin{align*}
    \delta_1=1, \hspace{1mm} \delta_k^2=1+\rho_{\alpha,\mathcal{O}}\sum_{i=1}^{k-1}\delta_i^2.
\end{align*}
\end{theorem}
\begin{proof}
By Theorem \ref{minkowskigcd} and the fact that $K^D$ is left unit reducible, we must have that the norm of $\mathbf{b}_1$ corresponds to the first successive minima. Denote by $\Lambda_{k-1}$ the sublattice of $\Lambda$ generated by the left-linear span of $B_{k-1}=\{\mathbf{b}_1,\dots,\mathbf{b}_{k-1}\}$ over $\mathcal{O}$, for some $2 \leq k \leq d$. By the definition of the successive minima, there exist $k$ linearly independent lattice vectors $\mathbf{s}_1,\dots,\mathbf{s}_k$ such that $q_{\alpha}(\varphi(\mathbf{s}_i))=\lambda_i^2$, $1 \leq i \leq k$. By the pigeonhole principle, there must exist at least one $\mathbf{s}_j$ such that $\mathbf{s}_j \not\in \Lambda_{k-1}$. However, there must exist a lattice vector $\mathbf{v}_j^* \in \Lambda$ so that $\{\mathbf{b}_1,\mathbf{b}_2,\dots,\mathbf{b}_{k-1},\mathbf{v}_j^*\}$ forms a primitive system for a sublattice containing $\mathbf{s}_j$, and therefore by proposition \ref{primitive}, this set must also be extendable to a basis for $\Lambda$, and so
\begin{align*}
    \mathbf{s}_j=\sum_{i=1}^{k-1}x_i \mathbf{b}_i + l\mathbf{v}_j^*,
\end{align*}
where $x_i, l \in \mathcal{O}$. Decompose $\mathbf{v}_j^*=\mathbf{p}+\mathbf{q}$, where the vector $\mathbf{q}$ is orthogonal to the space $\text{Span}_K(B_{k-1})$, which must be a nonzero vector. Then
\begin{align*}
    q_\alpha(\varphi(\mathbf{s}_j))=\lambda_j^2=q_\alpha\left(\varphi \left(\sum_{i=1}^{k-1}x_i\mathbf{b}_i+l\mathbf{p}\right)\right)+q_\alpha\left(\varphi(l\mathbf{q})\right)=q_\alpha\left(\varphi \left(\sum_{i=1}^{k-1}x_i\mathbf{b}_i+l\mathbf{p}\right)\right)+q_\alpha\left(\varphi((lu^{-1})(u\mathbf{q}))\right),
\end{align*}
where $u \in \mathcal{O}^\times$ is chosen to minimise the function $q_\alpha(\varphi(u\mathbf{q}))$. By the assumption that $K^D$ is left unit reducible, 
\begin{align*}
    \lambda_j^2 \geq q_\alpha(\varphi((lu^{-1})(u\mathbf{q}))) \geq q_{\alpha}(\varphi(u\mathbf{q})).
\end{align*}
Suppose that $u\mathbf{p}=\sum_{i=1}^{k-1}p_i\mathbf{b}_i$, for some $p_i \in K$. Using the orthogonalisation process detailed in lemma \ref{orthogonal}, by choosing a $\mathbf{y} \in \Lambda_{k-1}$ carefully so that
\begin{align*}
    q_\alpha(\varphi(u\mathbf{p}-\mathbf{y}))=q_\alpha\left(\sum_{i=1}^{k-1}m_i\mathbf{b}_i(i)\right)
\end{align*}
where $m_i \in K_{\mathbb{R}}$ such that $q_{\alpha}(m_i) \leq \rho_{\alpha,\mathcal{O}}$. Since $\mathbf{v}_j^*$ is extendable to a basis vector, we have
\begin{align*}
    &q_\alpha(\varphi(\mathbf{b}_k)) \leq q_\alpha(\varphi(u\mathbf{v}_j^*-\mathbf{y})) =q_\alpha(\varphi(u\mathbf{p}-\mathbf{y}))+q_\alpha(\varphi(u\mathbf{q}))
    \\& \leq q_\alpha\left(\sum_{i=1}^{k-1}m_i\mathbf{b}_i(i)\right)+\lambda_k^2 \leq \sum_{i=1}^{k-1}q_{\alpha}(m_i)q_{\alpha}(\mathbf{b}_i(i))+\lambda_k^2 \leq \rho_{\alpha,\mathcal{O}}\sum_{i=1}^{k-1}q_\alpha(\mathbf{b}_i(i))+\lambda_k^2
    \\& \leq \rho_{\alpha,\mathcal{O}}\sum_{i=1}^{k-1}q_\alpha(\varphi(\mathbf{b}_i))+\lambda_k^2\leq \rho_{\alpha,\mathcal{O}}\sum_{i=1}^{k-1}\delta_i^2\lambda_i^2+\lambda_k^2\leq \left(\rho_{\alpha,\mathcal{O}}\sum_{i=1}^{k-1}\delta_i^2+1\right)\lambda_k^2,
\end{align*}
as required.
\end{proof}
An exponential upper bound on the length of the basis vectors in terms of the successive minima immediately follows. By definition, we have
\begin{align*}
    &\delta_k^2=1+\rho_{\alpha,\mathcal{O}}\sum_{i=1}^{k-1}\delta_i^2=\left(1+\rho_{\alpha,\mathcal{O}}\sum_{i=1}^{k-2}\delta_i^2\right)+\rho_{\alpha,\mathcal{O}}\delta_{k-1}^2=\left(1+\rho_{\alpha,\mathcal{O}}\right)\delta_{k-1}^2,
\end{align*}
and so applying this recursively, and using the fact that $\delta_1=1$,
\begin{align*}
    \delta_k^2=\left(1+\rho_{\alpha,\mathcal{O}}\right)^{k-1},
\end{align*}
for all $1 \leq k \leq d$. Of course, for many fields it will hold that this bound is not optimal. For example, the authors have shown in previous work that for Euclidean imaginary quadratic fields, the vectors corresponding to the first two successive minima are always extendable to a basis for the lattice \cite{algebraicLLL}, and as such the exponent $k-1$ may be replaced with $k-l$, for some $l >1$.
\\ In fact, the bound can be improved for both imaginary quadratic and rational quaternion fields.
\begin{theorem}\label{minkowskiquaternion}
Let $K$ be a Euclidean imaginary quadratic field or a rational quaternion field, and let $\mathcal{M}(K)$ denote the Euclidean minimum of $K$. Let $q: K_\mathbb{R}^D \to \mathbb{R}^+$ be the quadratic norm defined by $q(\mathbf{v})=\sum_{i=1}^Dv_iv_i^*$, for $\mathbf{v}=(v_1,\dots,v_D) \in K_{\mathbb{R}}^D$. Let $\Lambda$ be a lattice spanned over $\mathcal{O}$, where $\mathcal{O}$ is a maximal order of $K$, and denote by $B=\{\mathbf{b}_1,\dots,\mathbf{b}_d\}$ a Minkowski reduced basis for $\Lambda$, and let $\lambda_1,\dots,\lambda_d$ denote the successive minima of $\Lambda$ with respect to the norm $\sqrt{q(\mathbf{v})}$. Then for all $1 \leq k \leq d$,
\begin{align*}
    q(\mathbf{b}_k) \leq \delta_k^2\lambda_k^2,
\end{align*}
where
\begin{align}
    &\delta_k=1, \hspace{2mm} 1 \leq k \leq g, \nonumber\\
    &\delta_k^2=\max\left\{1,\mathcal{M}(K)\sum_{i=1}^{k-1}\delta_i^2+\frac{1}{\text{nrd}_{K/\mathbb{Q}}(l)}\right\}, k>g, \label{quaternionbound}
\end{align}
where $l$ is the smallest element with respect to $\text{nrd}_{K/\mathbb{Q}}$ that is not zero or a unit and $g=\max\left\{2,\left\lfloor \mathcal{M}(K)^{-1}\right\rfloor\right\}$ except for the case $K=\frac{(-2,-5)}{\mathbb{Q}}$ for which $g=1$.
\end{theorem}
\begin{proof}
The values for $g$ follows from the fact that the first $g$ vectors must have lengths corresponding to the first $g$ successive minima of the lattice (this is proved in \cite{algebraicLLL}, \cite{stern2}, and the author's PhD thesis). As in the proof of the previous theorem, we take the sublattice $\Lambda_{k-1}$ spanned by the basis $\{\mathbf{b}_1,\dots,\mathbf{b}_{k-1}\}$, and so there exists some $1 \leq j \leq k$ such that the vector whose length corresponds to $\lambda_j$ is not contained in $\Lambda_{k-1}$, call this vector $\mathbf{s}_j$. Then there exists a $\mathbf{v}_j^*$ so that the set $\{\mathbf{b}_1,\dots,\mathbf{b}_{k-1},\mathbf{v}_j^*\}$ for some lattice vector $\mathbf{v}_j^*$ is extendable to a basis for $\Lambda$, and
\begin{align*}
    \mathbf{s}_j=\sum_{i=1}^{k-1}x_i\mathbf{b}_i+l\mathbf{v}_j^*,
\end{align*}
for some $x_i,l \in \mathcal{O}$. Since the field we are considering is either imaginary quadratic or rational quaternion, it holds that $q(x\mathbf{v})=\nrd_{K/\mathbb{Q}}(x)q(\mathbf{v})$ for all $x \in K_\mathbb{R}, \mathbf{v} \in K_{\mathbb{R}}^D$. Let $\mathbf{v}_j^*=\mathbf{p}+\mathbf{q}$, where $\mathbf{q}$ is orthogonal to the space $\text{Span}_{K_\mathbb{R}}(B_{k-1})$. Then
\begin{align*}
    q(\mathbf{s}_j)=\lambda_j^2=q\left(\sum_{i=1}^{k-1}x_i\mathbf{b}_i+l\mathbf{p}\right)+q(l\mathbf{q}).
\end{align*}
It must hold that $l \neq 0$, and if $l \in \mathcal{O}^\times$ (the unit group of $\mathcal{O}$), the vector $\mathbf{s}_j$ would be extendable to a basis vector and so $q(\mathbf{b}_k) = \lambda_j^2 \leq \lambda_k^2$. It must therefore hold that $q(\mathbf{q}) \leq \frac{1}{\text{nrd}_{K/\mathbb{Q}}(l)}\lambda_j^2$. For all $x \in K_{\mathbb{R}}$, there exists a $y \in \mathcal{O}$ such that $\text{nrd}_{K/\mathbb{Q}}(x-y) \leq \mathcal{M}(K)$, and so the result follows by an identical argument as before.
\end{proof}
This bound is tight for certain fields.
\begin{proposition}
Suppose that $\mathcal{O}$ is the maximal order of a rational quaternion or imaginary quadratic field $K$, and let $l$ be the smallest element with respect to $\text{nrd}_{K/\mathbb{Q}}$ that is not zero or a unit. Then there exists a lattice of dimension $d=\text{nrd}_{K/\mathbb{Q}}(l)+1$ with a Minkowski reduced basis $B=\{\mathbf{b}_1,\dots,\mathbf{b}_d\}$ such that $q(\mathbf{b}_d)=\left(1+\frac{1}{\text{nrd}_{K/\mathbb{Q}}(l)}\right)\lambda_d^2$.
\end{proposition}
\begin{proof}
Consider the lattice with the basis \begin{align*}\mathbf{b}_1=(l,0,\dots,0)^T, \mathbf{b}_2=(0,l,0,\dots,0)^T, \dots, \mathbf{b}_{d-1}=(0,0,\dots,l,0)^T,\mathbf{b}_d=(1,1,\dots,1,1)^T,\end{align*} $\mathbf{b}_i \in K^d$. It is easy to see that the lattice is Minkowski reduced - the vectors $\mathbf{b}_1,\mathbf{b}_2,\dots,\mathbf{b}_{d-1}$ are orthogonal, and given that $\mathbf{l}$ is not a unit we must have $\text{nrd}_{K/\mathbb{Q}}(1+pl) \geq 1$ for all $p \in \mathcal{O}$, so we cannot reduce the length of vector $\mathbf{b}_d$ by taking the transformation $\mathbf{b}_d+\sum_{i=1}^{d-1}p_i\mathbf{b}_i$, $p_i \in \mathcal{O}$. However, we have $l\mathbf{b}_d-\sum_{i=1}^{d-1}\mathbf{b}_i=(0,0,\dots,0,l)^T$, which has length $\text{nrd}_{K/\mathbb{Q}}(l)<d$. Moreover, this vector is orthogonal to the vectors $\mathbf{b}_1,\dots,\mathbf{b}_{d-1}$, and so must be the final successive minima of the lattice. The proposition follows.
\end{proof}
For the fields $K=\mathbb{Q}(i),\frac{(-1,-1)}{\mathbb{Q}}$ we have $\mathcal{M}(K)=1/2$ and a minimal nonzero, nonunit element of norm $2$, and hence we have $\delta_3^2=1+\frac{1}{2}=\frac{3}{2}=\mathcal{M}(K)\sum_{i=1}^2\delta_i^2 + \frac{1}{2}$. Similarly, for $\mathbb{Q}(\sqrt{-3})$ we have $\mathcal{M}(K)=1/3$ and a minimal nonzero, nonunit element of norm $3$, and hence we have $\delta_3^2=1+\frac{1}{3}=\frac{4}{3}=\mathcal{M}(K)\sum_{i=1}^3\delta_i^2 + \frac{1}{3}$, showing that for these fields our bound is sharp. It is interesting that these are the Euclidean fields whose maximal orders have $\mathbb{Z}$--bases that can be described by the fundamental units.
\section{Korkin-Zolotarev Reduction}
\subsection{Hermite-Korkin-Zolotarev reduction}
The proof methods for the results in this section follow closely to that of \cite{HKZ}. We have established that, by the method in equation (\ref{orthogonal}), we may represent a lattice vector as the direct sum of its orthogonal components over the quadratic norm $q_{\alpha}$ for any $\alpha \in \mathcal{P}$.
\begin{definition}
Assume that $K$ is a division algebra that is right-Euclidean, and let $\Lambda$ be a $d$-dimensional lattice over a left unit--reducible order $\mathcal{O}$ with basis $B=\{\mathbf{b}_1,\dots,\mathbf{b}_d\}$. We say that $B$ is \emph{HKZ--reduced} with respect to the quadratic norm $q_{\alpha}$ if, for all $1 \leq k \leq d$, the following properties are satisfied:
\begin{itemize}
    \item $\min_{(x_k,x_{k+1},\dots,x_d) \in \mathcal{O}^{d-k+1} \setminus \{\mathbf{0}\}}q_\alpha\left(\sum_{i=k}^d \varphi(x_i)\mathbf{b}_i(i)\right)=q_{\alpha}(\mathbf{b}_k(k))$,
    \item For all $1 \leq i <k$, $\min_{y \in \mathcal{O}}q_{\alpha}(\mu_{k,i}-\varphi(y))=q_\alpha(\mu_{k,i})$.
\end{itemize}
\end{definition}
\begin{proposition}
If $\mathcal{O}$ is right-Euclidean and $K^D$ is left unit reducible, then for any $\alpha \in \mathcal{P}$, and some lattice $\Lambda$ with basis $B=\{\mathbf{b}_1,\dots,\mathbf{b}_d\}$, $\mathbf{b}_i \in K^D$, there exists a primitive system $B^\prime=\{\mathbf{b}_1^\prime,\dots, \mathbf{b}_d^\prime\}$ such that $B^\prime$ is HKZ reduced.
\end{proposition}
\begin{proof}
This holds since, for every set $\{x_1,\dots,x_d\}$ such that $x_i \in \mathcal{O}$, $\gcd(x_1,\dots,x_d)=1$, there exists a matrix in $M_{d \times d}(\mathcal{O})$ with $x_1,\dots,x_d$ in one of the rows whose determinant is a unit, and so the shortest lattice vector can always be extended to a basis for $\Lambda$. The second property can obviously be attained using unimodular transformations to the basis.
\end{proof}
As with classical lattices, HKZ reduced lattices in the algebraic setting exhibit many useful properties.
\begin{theorem}
Let $\Lambda$ be a lattice over $\mathcal{O}$ (that is right-Euclidean) with basis $B=\{\mathbf{b}_1,\dots,\mathbf{b}_d\}$ that is HKZ reduced with respect to the quadratic norm $q_{\alpha}$ for some $\alpha \in \mathcal{P}$. Then, for all $1 \leq k \leq d$, we have
\begin{align*}
    q_{\alpha}(\varphi(\mathbf{b}_k)) \leq \left(1+(k-1)\rho_{\alpha,\mathcal{O}}\right)\lambda_k^2.
\end{align*}
\end{theorem}
\begin{proof}
Let $\mathbf{s}_j$ be the lattice vector such that $q_\alpha(\varphi(\mathbf{s}_j))=\lambda_j^2$, $1 \leq j \leq d$. Denote by $\mathbf{s}_j(k) :=\varphi(\mathbf{s}_j)-\sum_{i=1}^{k-1}\langle \varphi(\mathbf{s}_j),\mathbf{b}_i(i) \rangle_{\alpha} (\langle \mathbf{b}_i(i), \mathbf{b}_i(i) \rangle_{\alpha})^{-1} \mathbf{b}_i(i)$, i.e. the vector $\varphi(\mathbf{s}_j)$ after being orthogonalised to the vectors $\{\varphi(\mathbf{b}_1),\varphi(\mathbf{b}_2),\dots, \varphi(\mathbf{b}_{k-1})\}$ with respect to the quadratic norm $q_{\alpha}$. Since the vectors $\mathbf{s}_1,\dots,\mathbf{s}_d$ are linearly independent over $K$, we must have that at least one of $\mathbf{s}_1(k),\mathbf{s}_2(k),\dots,\mathbf{s}_k(k)$ is nonzero. Moreover, since $\Lambda$ is HKZ reduced, we must have $q_{\alpha}(\mathbf{b}_k(k)) \leq q_{\alpha}(\mathbf{s}_j(k)) \leq \lambda_k^2$. Using the notation as in (\ref{orthogonal}) and the fact that $q_{\alpha}(\mu_{i,j}) \leq \rho_{\alpha,\mathcal{O}}$ for all $1 \leq j < i \leq d$,
\begin{align*}
    &q_{\alpha}(\varphi(\mathbf{b}_k)) =q_{\alpha}\left(\mathbf{b}_k(k)+\sum_{i=1}^{k-1}\mu_{k,i}\mathbf{b}_i(i)\right)=\sum_{i=1}^{k-1}q_\alpha(\mu_{k,i}\mathbf{b}_i(i))+q_{\alpha}(\mathbf{b}_k(k))
    \\& \leq \sum_{i=1}^{k-1}q_\alpha(\mu_{k,i})q_{\alpha}(\mathbf{b}_i(i))+q_{\alpha}(\mathbf{b}_k(k)) \leq \rho_{\alpha,\mathcal{O}}\sum_{i=1}^{k-1}\lambda_i^2+\lambda_k^2
    \\& \leq \left(1+(k-1)\rho_{\alpha,\mathcal{O}}\right)\lambda_k^2,
\end{align*}
as required.
\end{proof}
Tighter bounds can be given when $K$ is imaginary quadratic or a rational quaternion field.
\begin{theorem}
Let $K$ be an imaginary quadratic field or rational quaternion field with Euclidean minimum $\mathcal{M}(K)$ and maximal order $\mathcal{O}$. Let $\Lambda$ be a lattice spanned over $\mathcal{O}$ with basis $B=\{\mathbf{b}_1,\dots,\mathbf{b}_d\}$ and assume that $\Lambda$ is HKZ reduced. Then for all $1 \leq k \leq d$,
\begin{align*}
    q(\mathbf{b}_k) \leq \left(1+(k-1)\mathcal{M}(K)\right)\lambda_k^2,
\end{align*}
where $q$ is defined as in \ref{minkowskiquaternion}.
\end{theorem}
The proof of this follows identically to the general case, so we omit it here. Note that for the field $K=\mathbb{Q}(\sqrt{-3})$, letting $\kappa_i=\frac{q(\mathbf{b}_i)}{\lambda_i^2}$, using our method we have $\kappa_i \leq 1+\frac{i-1}{3}=\frac{i+2}{3}$. For classical (real) lattices, the upper bound for the $\kappa_i$ is given by $\frac{i+3}{4}$. However, under the canonical embedding a $d$-dimensional lattice over $\mathcal{O}$ corresponds to a $2d$-dimensional real lattice, and as such the upper bound for the basis vector lengths under the canonical embedding would yield $\kappa_i \leq \frac{(2i-1)+1}{4}=\frac{i}{2}$, and so the bound we have proven is clearly sharper.
\subsection{Upper bounds for the successive minima and covering radius of algebraic lattices}
\begin{definition}
Let $\Lambda$ be a lattice of dimension $d$ spanned over an order $\mathcal{O}$ of $K$ with basis $B$. We define the \emph{$\alpha$--dual} of $\Lambda$ for some $\alpha \in \mathcal{P}$ as the set $\{\mathbf{v} \in \text{Span}_K(B): q_\alpha(\varphi(\mathbf{v}),\varphi(\mathbf{w})) \in \mathbb{Z}, \forall \mathbf{w} \in \Lambda\}$.
\end{definition}
Since the function $q_\alpha(\cdot,\cdot)$ is linear over addition, it is clear by a similar argument to the classical case that the $\alpha$--dual of a lattice is also a lattice. We denote the $\alpha$--dual of $\Lambda$ by $\Lambda_\alpha^*$. We also write $\lambda_i(\Lambda)$ to mean the $i$th successive minima of $\Lambda$.
\begin{lemma}
For any $\mathcal{O}$-lattice $\Lambda$ with basis $B$, under the map $\varphi$ the lattice $\Lambda_{\alpha}^*$ has the representation
\begin{align*}
    \varphi(\Lambda_\alpha^*)=\varphi(\mathcal{D}(\mathcal{O}/\mathbb{Z})^{-1})^d(\varphi(B)\varphi(B)^\dagger)^{-1}\varphi(B)\varphi(\alpha^{-1}),
\end{align*}
where $\mathcal{D}(\mathcal{O}/\mathbb{Z})$ denotes the different of $\mathcal{O}$ over $\mathbb{Z}$ and $\dagger$ denotes the Hermitian conjugate of a matrix with respect to the canonical conjugate previously defined. Moreover, we have
\begin{align*}
    \det_{\alpha}(\Lambda_{\alpha}^*)=\text{disc}(\mathcal{O}/\mathbb{Z})^{-2d}\det_{\alpha}(\Lambda)^{-1},
\end{align*}
and
\begin{align*}
    (\Lambda_{\alpha}^*)_{\alpha}^*=\Lambda.
\end{align*}
\end{lemma}
\begin{proof}
This follows from a direct calculation, since $\text{nrd}_{K/\mathbb{Q}}(\mathcal{D}(\mathcal{O}/\mathbb{Z}))=\text{disc}(\mathcal{O}/\mathbb{Z})$.
\end{proof}
\begin{theorem}
Let $\Lambda$ be a lattice over $\mathcal{O}$ (that is right-Euclidean) with basis $B=\{\mathbf{b}_1,\dots,\mathbf{b}_d\}$, and let $\lambda_i(\Lambda)$ denote the $i$th successive minima of $\Lambda$ with respect to the quadratic norm $q_\alpha$, for some $\alpha \in \mathcal{P}$. Denote by $\mu(\Lambda)=\max_{\mathbf{x} \in \text{Span}_K(B)}\min_{\mathbf{y} \in \Lambda}q_\alpha(\varphi(\mathbf{x}-\mathbf{y}))$. Then for all $1 \leq k \leq d$,
\begin{align*}
    \lambda_k(\Lambda)^2\lambda_{d-k+1}(\Lambda_\alpha^*)^2 &\leq \left(\frac{(2d-k+1)k}{2}\rho_{\alpha,\mathcal{O}}^2-\rho_{\alpha,\mathcal{O}}+1\right){\gamma_{\alpha,\mathcal{O},d}^\nu}^2\text{disc}(\mathcal{O}/\mathbb{Z})^{\frac{-2}{nm^2}},
    \\\mu(\Lambda)^2\lambda_1(\Lambda_\alpha^*)^2 &\leq \rho_{\alpha,\mathcal{O}}\sum_{i=1}^d{\gamma_{\alpha,\mathcal{O},i}^\nu}^2\text{disc}(\mathcal{O}/\mathbb{Z})^{\frac{-2}{nm^2}},
\end{align*}
where $\gamma_{\alpha,\mathcal{O},d}^\nu=\max_{i \in \{1,\dots,d\}}\gamma_{\alpha,\mathcal{O},i}$.
\end{theorem}
\begin{proof}
Suppose $B=\{\mathbf{b}_1,\dots,\mathbf{b}_d\}$ is a basis for $\Lambda$ that is HKZ reduced. Denote by $\varphi(\Lambda)$ the set of all vectors in $\Lambda$ under the map $\varphi$. Let us define by $\Lambda^k$ the set of vectors $\sum_{i=k}^d\varphi(x_i)\mathbf{b}_i(i) \in K_{\mathbb{R}}^D$, for $x_i \in \mathcal{O}$. Then, it is clear that $\left(\Lambda^{k}\right)_{\alpha}^*$ is a subset of $\varphi(\Lambda_\alpha^*)$ for all $1 \leq k \leq d$, and so we must have $\lambda_1(\varphi(\Lambda_{\alpha}^*)) \leq \lambda_1\left(\left(\Lambda^{k}\right)_{\alpha}^*\right)$, for all $1 \leq k \leq d$, with respect to the norm induced by $q_\alpha$. Moreover, for any $d$-dimensional lattice $\Lambda$ spanned over the order $\mathcal{O}$, it holds that 
\begin{align*}
\lambda_1(\Lambda)^2\lambda_1(\Lambda_\alpha^*)^2 \leq \gamma_{\alpha,\mathcal{O},d}\det_\alpha(\Lambda)^{\frac{1}{dnm^2}}\gamma_{\alpha,\mathcal{O},d}\det_{\alpha}(\Lambda_\alpha^*)^{\frac{1}{dnm^2}}=\gamma_{\alpha,\mathcal{O},d}^2\text{disc}(\mathcal{O}/\mathbb{Z})^{\frac{-2}{nm^2}}.
\end{align*}
Hence, since $B$ is HKZ reduced,
\begin{align*}
     &q_{\alpha}(\varphi(\mathbf{b}_k))\lambda_1(\Lambda_{\alpha}^*)^2=\sum_{i=1}^{k-1}q_\alpha(\mu_{k,i}\mathbf{b}_i(i))\lambda_1(\Lambda_\alpha^*)^2+q_{\alpha}(\mathbf{b}_k(k))\lambda_1(\Lambda_\alpha^*)^2
    \\& \leq \rho_{\alpha,\mathcal{O}}\sum_{i=1}^{k-1}\lambda_1(\Lambda^i)^2\lambda_1((\Lambda^i)_{\alpha}^*)^2+\lambda_1(\Lambda^k)^2\lambda_1((\Lambda^k)_\alpha^*)^2
    \\&\leq \left(\rho_{\alpha,\mathcal{O}}(k-1)+1\right){\gamma_{\alpha,\mathcal{O},d}^\nu}^2\text{disc}(\mathcal{O}/\mathbb{Z})^{\frac{-2}{nm^2}}.
\end{align*}
It therefore also holds that 
\begin{equation}\label{lambdainequality}
    \lambda_k(\Lambda)^2 \lambda_1(\Lambda_{\alpha}^*)^2 \leq \left(\rho_{\alpha,\mathcal{O}}(k-1)+1\right){\gamma_{\alpha,\mathcal{O},d}^\nu}^2\text{disc}(\mathcal{O}/\mathbb{Z})^{\frac{-2}{nm^2}},
\end{equation}
since $\lambda_i(\Lambda)^2\leq \max_{1 \leq i \leq k}q_{\alpha}(\varphi(\mathbf{b}_i))$. Since we must also have $\lambda_{k}(\varphi(\Lambda_\alpha^*)) \leq \lambda_{k}((\Lambda^{j})_{\alpha}^*)$ for all $1 \leq j \leq d$,$k \leq d-j+1$, by the above, it holds that
\begin{align*}
    &q_{\alpha}(\varphi(\mathbf{b}_k))\lambda_{d-k+1}(\Lambda_{\alpha}^*)^2 =\sum_{i=1}^{k-1}q_\alpha(\mu_{k,i}\mathbf{b}_i(i))\lambda_{d-k+1}(\Lambda_\alpha^*)^2+q_{\alpha}(\mathbf{b}_k(k))\lambda_{d-k+1}(\Lambda_\alpha^*)^2
    \\& \leq \rho_{\alpha,\mathcal{O}}\sum_{i=1}^{k-1}\lambda_1(\Lambda^i)^2\lambda_{d-i+1}((\Lambda^i)_{\alpha}^*)^2+\lambda_1(\Lambda^k)^2\lambda_{d-k+1}((\Lambda^k)_\alpha^*)^2
    \\&\leq \rho_{\alpha,\mathcal{O}}  \sum_{i=1}^{k-1}\left(\rho_{\alpha,\mathcal{O}}(d-i)+1\right){\gamma_{\alpha,\mathcal{O},d-i+1}^\nu}^2\text{disc}(\mathcal{O}/\mathbb{Z})^{\frac{-2}{nm^2}}+\left(\rho_{\alpha,\mathcal{O}}(d-k)+1\right){\gamma_{\alpha,\mathcal{O},d-k+1}^\nu}^2\text{disc}(\mathcal{O}/\mathbb{Z})^{\frac{-2}{nm^2}}
    \\&\leq \left(\frac{(2d-k+1)k}{2}\rho_{\alpha,\mathcal{O}}^2-\rho_{\alpha,\mathcal{O}}+1\right){\gamma_{\alpha,\mathcal{O},d}^\nu}^2\text{disc}(\mathcal{O}/\mathbb{Z})^{\frac{-2}{nm^2}},
\end{align*}
which proves the first result, since
\begin{align*}
    \lambda_k(\Lambda)^2\lambda_{n-k+1}(\Lambda_\alpha^*)^2 &\leq \max\{q_\alpha(\varphi(\mathbf{b}_j)): j \in \{1,\dots,k\}\}\lambda_{n-k+1}(\Lambda_\alpha^*)^2
    \\&\leq  \max\{q_\alpha(\varphi(\mathbf{b}_j))\lambda_{n-j+1}(\Lambda_\alpha^*)^2: j \in \{1,\dots,k\}\}
    \\&\leq \max\left\{\left(\frac{(2d-j+1)j}{2}\rho_{\alpha,\mathcal{O}}^2-\rho_{\alpha,\mathcal{O}}+1\right){\gamma_{\alpha,\mathcal{O},d}^\nu}^2\text{disc}(\mathcal{O}/\mathbb{Z})^{\frac{-2}{nm^2}}\right\}
    \\&= \left(\frac{(2d-k+1)k}{2}\rho_{\alpha,\mathcal{O}}^2-\rho_{\alpha,\mathcal{O}}+1\right){\gamma_{\alpha,\mathcal{O},d}^\nu}^2\text{disc}(\mathcal{O}/\mathbb{Z})^{\frac{-2}{nm^2}},
\end{align*}
when $k < \frac{d+1}{2}$, which can be assumed without loss of generality. As in the proof for the classical case, we prove the second result via induction. The claim clearly holds for $d=1$, so we assume $d>1$. For all $\mathbf{x}^\prime \in \text{Span}_K(\mathbf{b}_1)$, there clearly exists a $y \in \mathcal{O}$ such that
\begin{align*}
    q_\alpha(\varphi(\mathbf{x}^\prime-y\mathbf{b}_1)) \leq \rho_{\alpha,\mathcal{O}}q_\alpha(\varphi(\mathbf{x}^\prime)) \leq \rho_{\alpha,\mathcal{O}}\lambda_1(\Lambda_1)^2.
\end{align*}
By considering the sublattice spanned by the basis $\{\mathbf{b}_2,\mathbf{b}_3,\dots,\mathbf{b}_d\}$ (call this lattice $\Lambda^{\star}$), there exists a $\mathbf{v}^{\star} \in \Lambda^{\star}$ such that, for any $\mathbf{x}^{\star}$ in the orthogonal complement of $\text{Span}_K(\mathbf{b}_1)$ (with respect to $q_\alpha$) such that $q_{\alpha}(\varphi(\mathbf{x}^\star-\mathbf{v}^\star))\leq \mu(\Lambda^\star)^2$. Letting $\mathbf{x}=\mathbf{x}^\prime+\mathbf{x}^\wedge$ where $\mathbf{x} \in \text{Span}_K(\mathbf{b}_1,\dots,\mathbf{b}_d)$, $\mathbf{x}^\prime \in \text{Span}_K(\mathbf{b}_1)$ and $\mathbf{x}^\wedge=\mathbf{x}^\star-\mathbf{v}^\star$, there exists a $y \in \mathcal{O}$ such that
\begin{align*}
    q_{\alpha}(\varphi(\mathbf{x}-y\mathbf{b}_1))=q_{\alpha}(\varphi(\mathbf{x}^\prime-y\mathbf{b}_1))+q_{\alpha}(\varphi(\mathbf{x}^\wedge)) \leq \rho_{\alpha,\mathcal{O}}\lambda_1(\Lambda)^2+\mu(\Lambda^\star)^2.
\end{align*}
By this argument, it must hold that
\begin{align*}
    \lambda_1(\Lambda)^2\mu(\Lambda_\alpha^*)^2&\leq \rho_{\alpha,\mathcal{O}}\lambda_1(\Lambda)^2\lambda_1(\Lambda_\alpha^*)^2+\lambda_1(\Lambda)^2\mu((\Lambda^\star)^*)^2\leq \rho_{\alpha,\mathcal{O}}\lambda_1(\Lambda)^2\lambda_1(\Lambda_\alpha^*)^2+\lambda_1(\Lambda^\star)^2\mu((\Lambda^\star)^*)^2\\&\leq \rho_{\alpha,\mathcal{O}}\text{disc}(\mathcal{O}/\mathbb{Z})^{\frac{-2}{nm^2}}\sum_{i=1}^d\gamma_{\alpha,\mathcal{O},i}^2,
\end{align*}
where the final inequality is derived from (\ref{lambdainequality}) and the induction hypothesis.
\end{proof}
The bounds sharpen for imaginary quadratic and rational quaternion lattices.
\begin{theorem}
Let $K$ be an imaginary quadratic or rational quaternion field that is also Euclidean with Euclidean minimum $\mathcal{M}(K)$. Suppose that $\Lambda$ is a $d$--dimensional lattice spanned over a maximal order $\mathcal{O}$ of $K$ with successive minima $\lambda_i(\Lambda)$, $1 \leq i \leq d$ with respect to the quadratic norm $q:K^D \to \mathbb{R}^+, q(\mathbf{x})=\sum_{i=1}^D\varphi(x_i)\varphi(x_i)^*$. Denote by $\Lambda^*$ the dual of $\Lambda$, and $\mu$ similarly to before with respect to the quadratic norm $q$. Then for all $1 \leq k \leq d$, we have
\begin{align*}
    \lambda_k(\Lambda)^2\lambda_{n-k+1}(\Lambda^*)^2 &\leq \left(\frac{(2d-k+1)k}{2}\mathcal{M}(K)^2-\mathcal{M}(K)+1\right){\gamma_{1,\mathcal{O},d}^\nu}^2\text{disc}(\mathcal{O}/\mathbb{Z})^{\frac{-2}{nm^2}},
    \\ \mu(\Lambda)\lambda_1(\Lambda^*) &\leq \mathcal{M}(K)\sum_{i=1}^d{\gamma_{1,\mathcal{O},i}^\nu}^2\text{disc}(\mathcal{O}/\mathbb{Z})^{\frac{-2}{nm^2}}.
\end{align*}
\end{theorem}
\subsection{Block-Korkin-Zolotarev reduction}
The proof of the results in this section follow closely to that of \cite{schnorr}. We will continue to assume throughout that all orders treated are Euclidean domains.
\begin{definition}
A basis $B$ for a $d$-dimensional lattice $\Lambda$ spanned over the order $\mathcal{O}$ is said to be \emph{$\beta$-BKZ reduced} with respect to the quadratic norm $q_\alpha$, for some integer $2 \leq \beta \leq d$, if the following properties are satisfied:
\begin{itemize}
    \item For all $1 \leq k \leq d+1-\beta$, $k \leq j \leq k+\beta-1$, $\min_{(x_k,x_{k+1},\dots,x_{k+\beta-1}) \in \mathcal{O}^{\beta-1} \setminus \{\mathbf{0}\}}q_\alpha\left(\sum_{i=j}^{\beta+k-1}\varphi(x_i)\mathbf{b}_i(j)\right) \geq q_\alpha(\mathbf{b}_j(j))$,
    \item For all $1 \leq i <k \leq d$, $\min_{y \in \mathcal{O}}q_{\alpha}(\mu_{k,i}-\varphi(y))=q_\alpha(\mu_{k,i})$.
\end{itemize}
\end{definition}
BKZ reduction can be thought of as a generalisation of HKZ reduction. Note that a $d$--BKZ reduced basis is HKZ reduced, and that any $\beta$--BKZ reduced basis is also $(\beta-1)$--BKZ reduced.
\begin{proposition}[\cite{cdaminimum}, Proposition 3.1]
Let $\alpha \in \mathcal{P}$. Then we have
\begin{align*}
    \text{nrd}_{K/\mathbb{Q}}(\langle \mathbf{b}_i(i),\mathbf{b}_i(i) \rangle_\alpha) \leq \left(\frac{q_\alpha(\mathbf{b}_i(i))}{nm}\right)^{nm}.
\end{align*}
\end{proposition}
\begin{proposition}\label{BKZbound}
Let $\Lambda$ be a $\beta$--BKZ reduced lattice of dimension $d$ over $\mathcal{O}$, with basis $B=\{\mathbf{b}_1,\dots,\mathbf{b}_d\}$ and successive minima $\lambda_1,\dots,\lambda_d$ with respect to the quadratic norm $q_\alpha$, $\alpha \in \mathcal{P}$. Then 
\begin{align*}
q_\alpha(\varphi(\mathbf{b}_1))\leq \left(\frac{\gamma_{\alpha,\mathcal{O},\beta}}{nm}\right)^{2\frac{d-1}{\beta-1}}\lambda_1^2.
\end{align*}
\end{proposition}
\begin{proof}
Define the vectors $\mathbf{b}_{-\beta+3},\mathbf{b}_{-\beta+2},\dots,\mathbf{b}_0 \in K^D$ such that $q_\alpha(\varphi(\mathbf{b}_i))=q_\alpha(\varphi(\mathbf{b}_1))$ and $\langle \varphi(\mathbf{b}_i),\varphi(\mathbf{b}_j) \rangle_\alpha=0$ for all $-\beta+3 \leq i \leq 0, -\beta+3 \leq j \leq d$, $i \neq j$. It is easily deduced that the lattice with basis $\{\mathbf{b}_{-\beta+3},\dots,\mathbf{b}_0,\mathbf{b}_1,\dots,\mathbf{b}_d\}$ is also $\beta$--BKZ reduced. It is also clear that we must have
$\det_\alpha(\Lambda)=\prod_{i=1}^d \text{nrd}_{K/\mathbb{Q}}(\langle \mathbf{b}_i(i),\mathbf{b}_i(i) \rangle_\alpha)^m$. Then, by the definition of $\gamma_{\alpha,\mathcal{O},\beta}$ and the assumption that the basis is $\beta$--BKZ reduced, for all $-\beta+3 \leq k \leq d+1-\beta$,
\begin{align*}
    &q_{\alpha}(\mathbf{b}_k(k))^{nm^2\beta} \leq \gamma_{\alpha,\mathcal{O},\beta}^{nm^2\beta}\prod_{i=k}^{k+\beta-1}\text{nrd}_{K/\mathbb{Q}}(\langle \mathbf{b}_i(i),\mathbf{b}_i(i) \rangle_\alpha)^m \\ \implies &q_\alpha(\mathbf{b}_k(k))^{nm\beta} \leq \gamma_{\alpha,\mathcal{O},\beta}^{nm\beta}\prod_{i=k}^{k+\beta-1}\text{nrd}_{K/\mathbb{Q}}(\langle \mathbf{b}_i(i),\mathbf{b}_i(i) \rangle_\alpha).
\end{align*}
Multiplying through for $-\beta+3 \leq k \leq d+1-\beta$:
\begin{align*}
    \prod_{k=-\beta+3}^{d+1-\beta}q_{\alpha}(\mathbf{b}_k(k))^{nm\beta}&\leq \gamma_{\alpha,\mathcal{O},\beta}^{nm\beta(d-1)}\prod_{i=-\beta+3}^1\text{nrd}_{K/\mathbb{Q}}(\langle \mathbf{b}_i(i),\mathbf{b}_i(i)\rangle_\alpha)^{i+\beta-2}\prod_{i=2}^{d+1-\beta}\text{nrd}_{K/\mathbb{Q}}(\langle \mathbf{b}_i(i),\mathbf{b}_i(i)\rangle_\alpha)^\beta
    \\& \times \prod_{i=d+2-\beta}^d\text{nrd}_{K/\mathbb{Q}}(\langle \mathbf{b}_i(i),\mathbf{b}_i(i)\rangle_\alpha)^{d-i+1}
    \\&\leq \gamma_{\alpha,\mathcal{O},\beta}^{nm\beta(d-1)}\prod_{i=-\beta+3}^1\left(\frac{q_\alpha(\mathbf{b}_i(i))}{nm}\right)^{(i+\beta-2)nm}\prod_{i=2}^{d+1-\beta}\left(\frac{q_\alpha(\mathbf{b}_i(i))}{nm}\right)^{nm\beta}
    \\ &\times \prod_{i=d+2-\beta}^d\left(\frac{q_\alpha(\mathbf{b}_i(i))}{nm}\right)^{nm(d-i+1)}.
\end{align*}
Since $q_{\alpha}(\mathbf{b}_{-\beta+3}(-\beta+3))=q_\alpha(\mathbf{b}_{-\beta+2}(-\beta+2))=\dots=q_{\alpha}(\mathbf{b}_0(0))=q_\alpha(\mathbf{b}_1(1))$, balancing the inequality yields
\begin{align*}
    q_\alpha(\varphi(\mathbf{b}_1))^{\frac{nm\beta(\beta-1)}{2}} &\leq \left(\frac{\gamma_{\alpha,\mathcal{O},\beta}}{nm}\right)^{nm\beta(d-1)}\prod_{i=d+2-\beta}^dq_\alpha(\mathbf{b}_i(i)) \\&\leq \left(\frac{\gamma_{\alpha,\mathcal{O},\beta}}{nm}\right)^{nm\beta(d-1)}\max_{i \in \{d+2-\beta,d+1-\beta,\dots,d\}}q_{\alpha}(\mathbf{b}_i(i))^{\frac{nm\beta(\beta-1)}{2}},
\end{align*}
which implies that
\begin{align*}
    q_\alpha(\varphi(\mathbf{b}_1))\leq \left(\frac{\gamma_{\alpha,\mathcal{O},\beta}}{nm}\right)^{2\frac{d-1}{\beta-1}}\max_{i \in \{d+2-\beta,d+1-\beta,\dots,d\}}q_{\alpha}(\mathbf{b}_i(i)).
\end{align*}
If $q_\alpha(\varphi(\mathbf{b}_1))=\lambda_1^2$, we are done, so we assume that the shortest nonzero lattice vector $\mathbf{v}$ is different from $\mathbf{b}_1$. We may assume that $\mathbf{v}$ is not in the sublattice generated by the basis $\{\mathbf{b}_1,\dots,\mathbf{b}_{d-1}\}$ without loss of generality, as otherwise, we may take $d=d-1$ and continue applying the argument inductively until this holds. Denote by $\Lambda_i$ the lattice spanned by the basis $\{\mathbf{b}_i(i),\dots,\mathbf{b}_d(i)\}$, $d-\beta+1 \leq i \leq d$. Then, since $\Lambda$ is $\beta$-BKZ reduced, we have
\begin{align*}
    \lambda_1(\Lambda)^2=q_\alpha(\varphi(\mathbf{v})) \geq \lambda_1(\Lambda_i)^2 = q_\alpha(\mathbf{b}_i(i)).
\end{align*}
It therefore holds that $\lambda_1(\Lambda)^2 \geq \max_{i \in \{d+2-\beta,d+1-\beta,\dots,d\}}q_{\alpha}(\mathbf{b}_i(i)),$ which proves the claim.
\end{proof}
\begin{theorem}
Let $\Lambda$ be a $d$-dimensional lattice, and let $B=\{\mathbf{b}_1,\dots,\mathbf{b}_d\}$ be a $\beta$-BKZ reduced basis of $\Lambda$. Then for all $2 \leq i \leq d$, 
\begin{align*}
    &q_\alpha(\varphi(\mathbf{b}_1))\leq \left(\frac{\gamma_{\alpha,\mathcal{O},\beta}}{nm}\right)^{2\frac{d-1}{\beta-1}}\lambda_1^2,
    \\&q_\alpha(\varphi(\mathbf{b}_i)) \leq \left(\frac{\gamma_{\alpha,\mathcal{O},\beta}}{nm}\right)^{2\frac{d-i}{\beta-1}}\left(1+\rho_{\alpha,\mathcal{O}}\left(\frac{\gamma_{\alpha,\mathcal{O},\beta}}{nm}\right)^{\frac{i-1}{\beta-1}}\frac{1-\left(\frac{\gamma_{\alpha,\mathcal{O},\beta}}{nm}\right)^{2\frac{1-i}{\beta-1}}}{1-\left(\frac{\gamma_{\alpha,\mathcal{O},\beta}}{nm}\right)^{\frac{-2}{\beta-1}}}\right)\lambda_i^2.
\end{align*}
\end{theorem}
\begin{proof}
The first inequality was proved in \ref{BKZbound}. By applying \ref{BKZbound} to the lattice $L_i$ generated by the basis $\{\mathbf{b}_i(i),\dots,\mathbf{b}_d(i)\}$, we have
\begin{align*}
    q_\alpha(\mathbf{b}_i(i)) \leq \left(\frac{\gamma_{\alpha,\mathcal{O},\beta}}{nm}\right)^{2\frac{d-i}{\beta-1}}\lambda_1(\Lambda_i)^2\leq \left(\frac{\gamma_{\alpha,\mathcal{O},\beta}}{nm}\right)^{2\frac{d-i}{\beta-1}}\lambda_i^2,
\end{align*}
where the rightmost inequality comes from the fact that there are $i$ linearly independent vectors $\mathbf{v}_1,\dots,\mathbf{v}_i$ such that $q_\alpha(\varphi(\mathbf{v}_i)) \leq \lambda_i^2$, and so at least one of these must be such that $\mathbf{v}_k(i) \neq \mathbf{0}$ for some $1 \leq k \leq i$. Then
\begin{align*}
    q_\alpha(\varphi(\mathbf{b}_i))&=q_\alpha(\mathbf{b}_i(i))+\sum_{j=1}^iq_\alpha(\mu_{i,j}\mathbf{b}_j(j)) \leq q_\alpha(\mathbf{b}_i(i))+\sum_{j=1}^{i-1}q_\alpha(\mu_{i,j})q_\alpha(\mathbf{b}_j(j))
    \\& \leq \left(\frac{\gamma_{\alpha,\mathcal{O},\beta}}{nm}\right)^{2\frac{d-i}{\beta-1}}\lambda_i^2+\rho_{\alpha,\mathcal{O}}\sum_{j=1}^{i-1}\left(\frac{\gamma_{\alpha,\mathcal{O},\beta}}{nm}\right)^{2\frac{d-j}{\beta-1}}
    \\&\leq \left(\frac{\gamma_{\alpha,\mathcal{O},\beta}}{nm}\right)^{2\frac{d-i}{\beta-1}}\left(1+\rho_{\alpha,\mathcal{O}}\sum_{j=1}^{i-1}\left(\frac{\gamma_{\alpha,\mathcal{O},\beta}}{nm}\right)^{2\frac{i-j}{\beta-1}}\right)\lambda_i^2
    \\&=\left(\frac{\gamma_{\alpha,\mathcal{O},\beta}}{nm}\right)^{2\frac{d-i}{\beta-1}}\left(1+\rho_{\alpha,\mathcal{O}}\left(\frac{\gamma_{\alpha,\mathcal{O},\beta}}{nm}\right)^{\frac{i-1}{\beta-1}}\frac{1-\left(\frac{\gamma_{\alpha,\mathcal{O},\beta}}{nm}\right)^{2\frac{1-i}{\beta-1}}}{1-\left(\frac{\gamma_{\alpha,\mathcal{O},\beta}}{nm}\right)^{\frac{-2}{\beta-1}}}\right)\lambda_i^2,
\end{align*}
as required.
\end{proof}
Similarly to the case of HKZ reduction, we get sharper bounds when the division algebra over which we define the algebraic lattice is either imaginary quadratic or a rational quaternion field.
\begin{theorem}
Let $\Lambda$ be a lattice of dimension $d$ over $\mathcal{O}$ where $\mathcal{O}$ is the maximal order of an imaginary quadratic or rational quaternion field $K$, with $\beta-BKZ$ basis $\{\mathbf{b}_1,\dots,\mathbf{b}_d\}$, reduced with respect to the regular quadratic norm $q$. Then, for all $2 \leq i \leq d$, we have
\begin{align*}
    &q(\mathbf{b}_1) \leq \left(\frac{\gamma_{1,\mathcal{O},d}}{nm}\right)^{2\frac{d-1}{\beta-1}}\lambda_1^2,\\
    &q(\mathbf{b}_i) \leq  \left(\frac{\gamma_{1,\mathcal{O},\beta}}{nm}\right)^{2\frac{d-i}{\beta-1}}\left(1+\mathcal{M}(K)\left(\frac{\gamma_{1,\mathcal{O},\beta}}{nm}\right)^{\frac{i-1}{\beta-1}}\frac{1-\left(\frac{\gamma_{1,\mathcal{O},\beta}}{nm}\right)^{2\frac{1-i}{\beta-1}}}{1-\left(\frac{\gamma_{1,\mathcal{O},\beta}}{nm}\right)^{\frac{-2}{\beta-1}}}\right)\lambda_i^2.
\end{align*}
\end{theorem}
\section{The Closest Vector Problem}
\begin{theorem}
Let $\Lambda$ be a lattice over a division algebra $K$ with basis $B=\{\mathbf{b}_1,\dots,\mathbf{b}_d\}$, and suppose that $B$ is Minkowski reduced with respect to the quadratic norm $q_\alpha$, $\alpha \in \mathcal{P}$. Let $\mathbf{x}$ be an arbitrary point in $\text{Span}_K\{\mathbf{b}_1,\dots,\mathbf{b}_d\}$ so that $\varphi(\mathbf{x})=\sum_{i=1}^dx_i\mathbf{b}_i(i)$, $x_i \in K_{\mathbb{R}}$. Suppose $\mathbf{v}=\sum_{i=1}^d v_i\mathbf{b}_i \in \Lambda, v_i \in \mathcal{O},$ is constructed so that $q_{\alpha}\left(x_j-\sum_{i=j}^d\varphi(v_i)\mu_{i,j}\right) \leq \rho_{\alpha,\mathcal{O}}$ for all $1 \leq j \leq d$ (where we set $\mu_{j,j}=1$), or $q\left(x_j-\sum_{i=j}^d\varphi(v_i)\mu_{i,j}\right) \leq \mathcal{M}(K)$ if $K$ is imaginary quadratic or a rational quaternion ring, where $q$ defines the regular quadratic norm as previously defined. Then
\begin{align*}
    q_\alpha(\varphi(\mathbf{v}-\mathbf{x})) \leq (1+\rho_{\alpha,\mathcal{O}})\delta_d^2\lambda_d^2,
\end{align*}
where $\delta_i^2$ is defined as in Theorem \ref{minkowskibounds1}. If the field is imaginary quadratic or quaternion, we may replace $\rho_{\alpha,\mathcal{O}}$ with $\mathcal{M}(K)$ and define $\delta_i^2$ according to Theorem \ref{minkowskiquaternion}.
\end{theorem}
\begin{proof}
We assume $\alpha=1$ if the field is imaginary quadratic or rational quaternion, and replace $\rho_{\alpha,\mathcal{O}}$ with $\mathcal{M}(K)$ when considering these fields. By definition, we have
\begin{align*}
    q_{\alpha}(\varphi(\mathbf{x}-\mathbf{v}))&=\sum_{i=1}^{d}q_\alpha\left(\left(x_j-\sum_{i=j}^d\varphi(v_i)\mu_{i,j}\right)\mathbf{b}_i(i)\right)\leq \sum_{i=1}^{d}q_\alpha\left(\left(x_j-\sum_{i=j}^d\varphi(v_i)\mu_{i,j}\right)\right)q_\alpha\left(\mathbf{b}_i(i)\right)
    \\&\leq \sum_{i=1}^d\rho_{\alpha,\mathcal{O}}q_\alpha(\varphi(\mathbf{b}_i)) \leq \sum_{i=1}^d\rho_{\alpha,\mathcal{O}}\delta_i^2\lambda_i^2 \leq (1+\rho_{\alpha,\mathcal{O}})\delta_d^2\lambda_d^2,
\end{align*}
as required.
\end{proof}
The following theorem may also be applied for HKZ reduced bases by setting $\beta=d$.
\begin{theorem}
Let $\Lambda$ be a lattice over a division algebra $K$ with basis $B=\{\mathbf{b}_1,\dots,\mathbf{b}_d\}$, and suppose that $B$ is $\beta$-BKZ reduced with respect to the quadratic norm $q_\alpha$, $\alpha \in \mathcal{P}$. Let $\mathbf{x}$ be an arbitrary point in $\text{Span}_K\{\mathbf{b}_1,\dots,\mathbf{b}_d\}$ so that $\varphi(\mathbf{x})=\sum_{i=1}^dx_i\mathbf{b}_i(i)$, $x_i \in K_{\mathbb{R}}$. Suppose $\mathbf{v}=\sum_{i=1}^d v_i\mathbf{b}_i \in \Lambda, v_i \in \mathcal{O},$ is constructed so that $q_{\alpha}\left(x_j-\sum_{i=j}^d\varphi(v_i)\mu_{i,j}\right) \leq \rho_{\alpha,\mathcal{O}}$ for all $1 \leq j \leq d$ (where we set $\mu_{j,j}=1$), or $q\left(x_j-\sum_{i=j}^d\varphi(v_i)\mu_{i,j}\right) \leq \mathcal{M}(K)$ if $K$ is imaginary quadratic or a rational quaternion ring, where $q$ defines the regular quadratic norm as previously defined. Then for $2 \leq \beta \leq d-1$, we have
\begin{align*}
    q_\alpha(\varphi(\mathbf{x}-\mathbf{v})) \leq \rho_{\alpha,\mathcal{O}}\left(\frac{\gamma_{\alpha,\mathcal{O},\beta}}{nm}\right)^{2\frac{d-1}{\beta-1}}\frac{1-\left(\frac{\gamma_{\alpha,\mathcal{O},\beta}}{nm}\right)^{2\frac{1-d}{\beta-1}}}{1-\left(\frac{\gamma_{\alpha,\mathcal{O},\beta}}{nm}\right)^{\frac{-2}{\beta-1}}}\lambda_d^2,
\end{align*}
and for $\beta=d$ i.e. HKZ reduced basis,
\begin{align*}
    q_\alpha(\varphi(\mathbf{x}-\mathbf{v})) \leq d\rho_{\alpha,\mathcal{O}}\lambda_d^2.
\end{align*}
If the field is imaginary quadratic or quaternion, we may replace $\rho_{\alpha,\mathcal{O}}$ with $\mathcal{M}(K)$.
\end{theorem}
\begin{proof}
The proof follows almost identically to before:
\begin{align*}
    q_{\alpha}(\varphi(\mathbf{x}-\mathbf{v}))&=\sum_{i=1}^{d}q_\alpha\left(\left(x_j-\sum_{i=j}^d\varphi(v_i)\mu_{i,j}\right)\mathbf{b}_i(i)\right)\leq \sum_{i=1}^{d}q_\alpha\left(\left(x_j-\sum_{i=j}^d\varphi(v_i)\mu_{i,j}\right)\right)q_\alpha\left(\mathbf{b}_i(i)\right)
    \\&\leq \sum_{i=1}^d\rho_{\alpha,\mathcal{O}}q_\alpha(\varphi(\mathbf{b}_i)) \leq \rho_{\alpha,\mathcal{O}}\sum_{i=1}^d\left(\frac{\gamma_{\alpha,\mathcal{O},\beta}}{nm}\right)^{2\frac{d-i}{\beta-1}}\lambda_i^2 \leq \rho_{\alpha,\mathcal{O}}\left(\frac{\gamma_{\alpha,\mathcal{O},\beta}}{nm}\right)^{2\frac{d-1}{\beta-1}}\frac{1-\left(\frac{\gamma_{\alpha,\mathcal{O},\beta}}{nm}\right)^{2\frac{1-d}{\beta-1}}}{1-\left(\frac{\gamma_{\alpha,\mathcal{O},\beta}}{nm}\right)^{\frac{-2}{\beta-1}}}\lambda_d^2.
\end{align*}
The proof for $\beta=d$ is easily deduced following the argument above.
\end{proof}
\section{Acknowledgements}
We would like to thank the Engineering and Physical Sciences Research Council for funding the author's PhD. 

\newpage
\appendixpage
We will now prove proposition \ref{unitreduciblefields}. Each of these fields is either real quadratic, or a real quadratic extension of an imaginary quadratic field. Since $q_{\alpha}(\varphi(x\mathbf{v}))=\trace_{K/\mathbb{Q}}(x\langle \mathbf{v},\mathbf{v} \rangle_\alpha x^*)$ for any $x \in K$, $\mathbf{v} \in K^D$, it suffices to treat the case $D=1$, $\alpha=1$ since $\langle \mathbf{v},\mathbf{v} \rangle_\alpha$ is an arbitrary element of $\mathcal{P}$. We will denote $\langle \mathbf{v},\mathbf{v} \rangle_\alpha=v$. For any $x \in K$, where $K$ is either real quadratic or a real quadratic extension of an imaginary quadratic field, we have
\begin{align*}
    \trace_{K/\mathbb{Q}}(xvx^*)= |x|^2v + |\sigma(x)|^2\sigma(v),
\end{align*}
where $\sigma$ is the field automorphism that fixes all elements except the irrational part of the real quadratic subfield, say $\mathbb{Q}(\sqrt{d})$, such that $\sigma(\sqrt{d})=-\sqrt{d}$. 
\\
First, assume that $K=\mathbb{Q}(\sqrt{5})$. The ring of integers of $K$ is $\mathcal{O}=\mathbb{Z}\left[\frac{1+\sqrt{5}}{2}\right]$, and the unit group $\mathcal{O}^\times$ is generated by the elements $-1,\phi=\frac{1+\sqrt{5}}{2}$ under multiplication. Let $x$ denote an arbitrary nonzero element of $\mathcal{O}$. We assume without loss of generality that $|x| \geq |\sigma(x)|$. Assume that $n$ is the largest integer such that $\phi^n \leq |x|$. Then we have
\begin{align*}
    \phi^n \leq |x| \leq \phi^{n+1}=\phi \phi^n.
\end{align*}
Let $X=|x||\sigma(x)|$. Since we are considering non-units only, we have $X \geq 2$, $X \in \mathbb{Z}$. Hence,
\begin{align*}
    |\sigma(x)|=\frac{X}{|x|} \geq \frac{X}{\phi^{n+1}}=\frac{X}{\phi}\phi^{-n}.
\end{align*}
We have $\phi<2$, and $X \geq 2$, and so $|\sigma(x)| > \phi^{-n}=|\sigma(\phi)|^n$. Therefore,
\begin{align*}
    \trace_{K/\mathbb{Q}}(xvx^*)=|x|^2v+|\sigma(x)|^2\sigma(v) > |\phi|^{2n}v+|\sigma(\phi)|^{2n}\sigma(v)=\trace_{K/\mathbb{Q}}(\phi^n v \phi^n),
\end{align*}
which proves the claim for $K=\mathbb{Q}(\sqrt{5})$.
\\
We will use the above method to prove the rest of the cases. Let $K=\mathbb{Q}(\sqrt{2})$. The ring of integers of $K$ is $\mathcal{O}=\mathbb{Z}[\sqrt{2}]$, and the unit group $\mathcal{O}^\times$ is generated by the elements $-1, u=1+\sqrt{2}$ under multiplication. Let $x$ be a nonzero element of $\mathcal{O}$ satisfying $|x| \geq |\sigma(x)|$. As before, set $n$ such that
\begin{align*}
    u^n \leq |x| \leq u^{n+1},
\end{align*}
and
\begin{align*}
    |x||\sigma(x)|=X.
\end{align*}
We assume that $x$ is not a unit and so $X \geq 2$. Then
\begin{align*}
    |\sigma(x)|=\frac{X}{|x|} \geq \frac{X}{u^{n+1}}=\frac{X}{u}u^{-n}.
\end{align*}
Since $u <3$, if we have $X \geq 3$ we are done, by a similar argument to before, so we assume that $X=2$. The only elements with norm $2$ are associates of $\sqrt{2}$, and since we clearly have $|\sqrt{2}|>|1|=1$, we are done for this case too. Hence, for every nonzero $x \in \mathcal{O}$ there exists a $u \in \mathcal{O}^\times$ satisfying $|u| \leq |x|, |\sigma(u)| < |\sigma(x)|$ and so the claim holds for $K=\mathbb{Q}(\sqrt{2})$ by a similar argument to before.
\\
Now suppose $K=\mathbb{Q}(\sqrt{3})$. The ring of integers of $K$ is $\mathcal{O}=\mathbb{Z}[\sqrt{3}]$, and the unit group $\mathcal{O}^\times$ is generated by the elements $-1,u=2+\sqrt{3}$ under multiplication. Let $x$ be a nonzero element of $\mathcal{O}$ satisfying $|x| \geq |\sigma(x)|$. As before, set $n$ such that
\begin{align*}
    u^n \leq |x| \leq u^{n+1},
\end{align*}
and
\begin{align*}
    |x||\sigma(x)|=X.
\end{align*}
We assume that $x$ is not a unit and so $X \geq 2$. Then
\begin{align*}
    |\sigma(x)|=\frac{X}{|x|} \geq \frac{X}{u^{n+1}}=\frac{X}{u}u^{-n}.
\end{align*}
Since $u<4$, if we have $X \geq 4$ we are done, so we assume that $X \in \{2,3\}$. If $X=3$, $x$ must be an associate of $\sqrt{3}$ and clearly $|\sqrt{3}| >|1|=1$, so we are done for this case. A solution for $X=2$ would mean that there exists a solution to the Pell's equation
\begin{align*}
    p^2-3q^2=2
\end{align*}
for some integers $p,q$. We must have either both $p,q$ are even or both are odd, as otherwise the right hand side would be an odd integer. Clearly they cannot be even, as otherwise $4 \mid p^2-3q^2$, so we must have that both $p,q$ are odd. Using the fact that if $p$ is odd we must have $p^2 \equiv 1 \mod 8$, we get
\begin{align*}
    p^2 \equiv 3q^2+2 \equiv 1 \mod 8 \iff q^2 \equiv 5 \mod 8,
\end{align*}
which is a contradiction since this would imply that $q^2$ is not odd (in fact, there exist no solutions). Hence, there exists no solution for $X=2$. Therefore, for every nonzero $x \in \mathcal{O}$ there exists a $u \in \mathcal{O}^\times$ satisfying $|u| \leq |x|, |\sigma(u)| < |\sigma(x)|$ and so the claim holds for $K=\mathbb{Q}(\sqrt{3})$ by a similar argument to before.
\\ Now suppose that $K=\mathbb{Q}(\zeta_8)=\mathbb{Q}(i,\sqrt{2})$. The ring of integers of $K$ is $\mathcal{O}=\mathbb{Z}[\zeta_8]$, and the unit group $\mathcal{O}^\times$ is generated by the elements $\zeta_8, u=1+\zeta_8+\zeta_8^2$ under multiplication. Let $x$ denote an arbitrary nonzero element of $\mathcal{O}$ satisfying $|x| \geq |\sigma(x)|$. As before, set $n$ such that
\begin{align*}
    |u|^{2n} \leq |x|^2 \leq |u|^{2(n+1)},
\end{align*}
and
\begin{align*}
    |x|^2|\sigma(x)|^2=X.
\end{align*}
We assume that $x$ is not a unit and so $X \geq 2$. Then
\begin{align*}
    |\sigma(x)|^2=\frac{X}{|x|^2} \geq \frac{X}{|u|^{2(n+1)}}=\frac{X}{|u|^2}|u|^{-2n}.
\end{align*}
Since $|u|^2=3+2\sqrt{2}<6$, if we have $X \geq 6$ we are done so we assume that $2 \leq X \leq 5$. First, we show that there exists no solution for $X=3,5$. Any element $x \in \mathcal{O}$ can be rewritten as $x=\sum_{i=0}^3x_i \zeta_8^i$, where $x_i \in \mathbb{Z}$, and so
\begin{align*}
    |x|^2=\sum_{i=0}^3x_i^2 +\sqrt{2}(x_0x_1+x_0x_3-x_1x_2+x_2x_3) \in \mathbb{Z}[\sqrt{2}],
\end{align*}
and so $\norm_{K/\mathbb{Q}}(x)=p^2-2q^2$ for some integers $p,q$. For either $X=3$ or $X=5$, we need that both $p,q$ are odd or both are even, and so we cannot have that both are even since $4$ does not divide $3$ or $5$. Hence, $p,q$ must both be odd, so we have
\begin{align*}
    p^2 \equiv 2q^2+X \equiv 1 \mod 8,
\end{align*}
which yields $q^2 \equiv 3 \mod 8$ or $q^2 \equiv 7 \mod 8$ when $X=3$, and $q^2 \equiv \pm 2 \mod 8$ when $X=5$, all of which are contradictions by the assumption that $q$ must be odd. Therefore there are no solutions for $X=3,5$. The solutions for $X=4$ are associates of the integer $1+\zeta_8^2$, but we clearly have $|1+\zeta_8^2|^2=2>|1|^2=1$, so we are done for this case. Finally, we focus on the case $X=2$, which constitutes associates of the values $1 \pm \zeta_8$. For this case, we cannot find a unit $u$ such that $|u| \leq |x|, |\sigma(u)| \leq |\sigma(x)|$. However, we note that
\begin{align*}
    2\trace_{K/\mathbb{Q}}((1\pm \zeta_8)v(1 \pm \zeta_8)^*)&=2|1\pm \zeta_8|^2v + 2|\sigma(1 \pm \zeta_8)|^2 \sigma(v)
    \\&=(4 \pm 2\sqrt{2})v+(4 \mp 2\sqrt{2})\sigma(v)\\&=(v+\sigma(v))+(3\pm 2\sqrt{2})v+(3 \mp 2\sqrt{2})\sigma(v)\\&=\trace_{K/\mathbb{Q}}(v)+\trace_{K/\mathbb{Q}}((1 \pm \zeta_8+\zeta_8^2)v(1 \pm \zeta_8+ \zeta_8^2)^*),
\end{align*}
and so $\trace_{K/\mathbb{Q}}((1\pm \zeta_8)v(1 \pm \zeta_8)^*) \geq \min\{\trace_{K/\mathbb{Q}}(v), \trace_{K/\mathbb{Q}}((1 \pm \zeta_8+\zeta_8^2)v(1 \pm \zeta_8+ \zeta_8^2)^*)\}$. Hence, for all nonzero $x \in \mathcal{O}$, it is possible to find a $u \in \mathcal{O}^\times$ such that
\begin{align*}
    \trace_{K/\mathbb{Q}}(uvu^*) \leq \trace_{K/\mathbb{Q}}(xvx^*),
\end{align*}
which proves the claim for $K=\mathbb{Q}(\zeta_8)$.
\\
Now, suppose that $K=\mathbb{Q}(\zeta_{12})=\mathbb{Q}(i,\sqrt{3})$. The ring of integers of $K$ is $\mathcal{O}=\mathbb{Z}[\zeta_{12}]$, and the unit group $\mathcal{O}^\times$ is generated by the elements $\zeta_{12}, 1+\zeta_{12}$ under multiplication. Let $x$ denote an arbitrary nonzero element of $\mathcal{O}$ satisfying $|x| \geq |\sigma(x)|$. As before, set $n$ such that
\begin{align*}
    |u|^{2n} \leq |x|^2 \leq |u|^{2(n+1)},
\end{align*}
and
\begin{align*}
    |x|^2|\sigma(x)|^2=X.
\end{align*}
We assume that $x$ is not a unit and so $X \geq 2$. Then
\begin{align*}
    |\sigma(x)|^2=\frac{X}{|x|^2} \geq \frac{X}{|u|^{2(n+1)}}=\frac{X}{|u|^2}|u|^{-2n}.
\end{align*}
Since $|u|^2=2+\sqrt{3}<4$, if we have $X \geq 4$ we are done so we assume that $2 \leq X \leq 3$. Any element $x \in \mathcal{O}$ can be rewritten as $x=\sum_{i=0}^3x_i\zeta_{12}^i$ where $x_i \in \mathbb{Z}$, so
\begin{align*}
    |x|^2=\sum_{i=0}^3x_i^2+x_0x_2+x_1x_3+\sqrt{3}(x_0x_1+x_2x_3),
\end{align*}
and so $\norm_{K/\mathbb{Q}}(x)=p^2-3q^2$ for some integers $p,q$. We have previously established that the solution for $X=2$ does not exist. Moreover, there exists no solution for $X=3$. This can be shown as follows: assume first that $p$ is odd. This means that $q$ must be even. Then we have
\begin{align*}
    p^2 \equiv 3q^2+3 \equiv 1 \mod 8 \iff q^2 \equiv 2 \mod 8,
\end{align*}
which is clearly a contradiction, as we need $q^2 \equiv 4 \mod 8$ or $q^2 \equiv 0 \mod 8$. Now assume that $p$ is even, which means that $q$ must be odd. Then
\begin{align*}
    p^2 \equiv 3q^2+3 \equiv 6 \mod 8,
\end{align*}
which again is impossible. Hence there exists no solution for $X=3$, since $X$ must be positive. Therefore, for all nonzero $x \in \mathcal{O}$ there exists a $u \in \mathcal{O}^\times$ satisfying $|u| \leq |x|$, $|\sigma(u)| < |\sigma(x)|$ and so the claim holds for $K=\mathbb{Q}(\zeta_{12})$ by a similar argument to before.
\\
Finally, we show that $K^D$, $K=\mathbb{Q}(\sqrt{6})$ is not unit reducible for any $D$ by constructing a counterexample. Consider $v=(2+\sqrt{6})^2 \in \mathcal{P}$. The ring of integers of $K$ is $\mathcal{O}=\mathbb{Z}[\sqrt{6}]$ and the unit group is generated by the elements $-1, u=5+2\sqrt{6}$ under multiplication.
\begin{lemma}\label{convex}
Denote by $f: \mathbb{R} \to \mathbb{R}$ the function $f(x)=e^{kx}v+e^{-kx}\sigma(v)$ for any $k \in \mathbb{R}^\times$, $v \in \mathcal{P}$. Then $f(x) \geq f(0)$ for all $|x| \geq 1$ if and only if $f(1) \geq f(0)$, $f(-1) \geq f(0)$.
\end{lemma}
\begin{proof}
Assume without loss of generality that $v \geq \sigma(v)$. We have
\begin{align*}
    f(x)=(v-\sigma(v))e^{kx}+\sigma(v)\cosh{x}.
\end{align*}
Both $e^x,\cosh{x}$ are convex functions, and so their linear sum must also be convex. The claim follows.
\end{proof}
Now, note that we have
\begin{align*}
    &\trace_{K/\mathbb{Q}}(uvu)=(5+2\sqrt{6})^2(2+\sqrt{6})^2+(5-2\sqrt{6})^2(2-\sqrt{6})^2=1940>20=\trace_{K/\mathbb{Q}}(v),
    \\&\trace_{K/\mathbb{Q}}(\sigma(u)v\sigma(u))=(5-2\sqrt{6})^2(2+\sqrt{6})^2+(5+2\sqrt{6})^2(2-\sqrt{6})^2=20=\trace_{K/\mathbb{Q}}(v).
\end{align*}
By Lemma \ref{convex}, it is impossible to find a unit $U$ such that $\trace_{K/\mathbb{Q}}(UvU)<\trace_{K/\mathbb{Q}}(v)$. However, we have
\begin{align*}
    \trace_{K/\mathbb{Q}}((2-\sqrt{6})v(2-\sqrt{6}))=2^3=8<\trace_{K/\mathbb{Q}}(v),
\end{align*}
and $\norm_{K/\mathbb{Q}}(2-\sqrt{6})=2$, so this constitutes a counterexample, proving that $K^D$ cannot be unit reducible for any $D$.
\end{document}